\documentclass[12pt]{article}%
\usepackage{amsfonts, amsmath, amssymb}
\usepackage{subfig}
\usepackage[dvips]{graphicx}
\usepackage{pst-plot, pstricks, pst-node}
\usepackage{cite}
\usepackage{amsmath}
\usepackage{amsfonts}
\usepackage{amssymb}%
\providecommand{\U}[1]{\protect\rule{.1in}{.1in}}
\newtheorem{theorem}{Theorem}

\newtheorem{definition}[theorem]{Definition}

\newtheorem{lemma}[theorem]{Lemma}

\newtheorem{proposition}[theorem]{Proposition}

\newenvironment{proof}[1][Proof]{\noindent\textbf{#1.} }{\ \rule{0.5em}{0.5em}}
\begin{document}

\title{Standing waves in near-parallel vortex filaments}
\author{Walter Craig\thanks{Department of Mathematics and Statistics, McMaster
University, Hamilton, Ontario L8S 4K1, Canada}, Carlos
Garc\'{\i}a-Azpeitia\thanks{Departamento de Matem\'{a}ticas, Facultad de
Ciencias, Universidad Nacional Aut\'{o}noma de M\'{e}xico, 04510 M\'{e}xico
DF, M\'{e}xico}, Chi-Ru Yang\footnotemark[1] }
\maketitle

\begin{abstract}
A model derived in \cite{KMD95} for $n$ near-parallel vortex filaments in a
three dimensional fluid region takes into consideration the pairwise
interaction between the filaments along with an approximation for motion by
self-induction. The same system of equations appears in descriptions of the
fine structure of vortex filaments in the Gross -- Pitaevski model of Bose --
Einstein condensates. In this paper we construct families of standing waves
for this model, in the form of $n$ co-rotating near-parallel vortex filaments
that are situated in a central configuration. This result applies to any pair
of vortex filaments with the same circulation, corresponding to the case
$n=2$. The model equations can be formulated as a system of Hamiltonian PDEs,
and the construction of standing waves is a small divisor problem. The methods
are a combination of the analysis of infinite dimensional Hamiltonian
dynamical systems and linear theory related to Anderson localization. The main
technique of the construction is the Nash-Moser method applied to a
Lyapunov-Schmidt reduction, giving rise to a bifurcation equation over a
Cantor set of parameters.
\end{abstract}

\section*{Introduction}

The dynamics of self-interacting infinitesimal vortex tubes in a three
dimensional fluid is in general a complex problem, involving the coherence of
structures under time evolution of individual filaments, and possible filament
collision and reconnection. However in certain limiting cases there are models
that are relatively straightforward to understand, which exhibit dynamics that
are subject to rigorous analysis. In \cite{KMD95}, a model system of equations
was derived for the interaction of $n$ near-parallel vortex filaments, all
sharing the same circulation. In this model we consider points in
$\mathbb{R}^{3}$ coordinatized by $(x_{1}+ix_{2},x_{3})\in\mathbb{C}%
\times\mathbb{R}$ and give $n$-many curves $(u_{j}(t,s),s)\in\mathbb{C}%
\times\mathbb{R}$ that describe the positions of $n$ vertically oriented
vortex filaments, each with unit circulation $\gamma=1$. From \cite{KMD95} the
system of model equations for the dynamics of $n$ near-parallel vortex
filaments is given by
\begin{equation}
\partial_{t}u_{j}=i\left(  \partial_{ss}u_{j}+\sum_{i=1,i\not =j}^{n}%
\frac{u_{j}-u_{i}}{\left\vert u_{j}-u_{i}\right\vert ^{2}}\right)
~\text{,}\quad j=1,\dots n~. \label{fp}%
\end{equation}
The case of exactly parallel vortex filaments in an incompressible inviscid
fluid reduces to a problem of interactions of point vortices in $\mathbb{R}%
^{2}$, which is described by a finite dimensional Hamiltonian system. The
model \eqref{fp} represents an approximation of near - parallel vortex
filament interactions, which is valid in an asymptotic regime in which vortex
filaments have small deviation from being exactly parallel and they remain
uniformly separated.

The above model has been extensively studied. In \cite{KPV03} the long time
existence of solutions of \eqref{fp} is given for $n=2$ and for certain near
triangular configurations with $n=3$. The article \cite{BaMi12} gives a long
time existence theorem for the case $n=4$ near the configuration of a square,
and gives a global in time solution for central configurations consisting of
rotating regular polygons of an arbitrary number $n\geq3$ of filaments. In
this latter work the authors give a uniform lower bound on the distance
between the filaments. In a very recent paper \cite{BaMi16} the same authors
give examples where the model \eqref{fp} evolves filaments which intersect in
finite time, thus giving rise to singular behavior; this illustrates the
limits of the validity of the model, at least as an approximation of the Euler
equations uniformly in time.

We recently learned of a second setting in which the near-parallel vortex
model appear, namely in studies of the dynamics of vortices for the Gross --
Pitaevski model of Bose -- Einstein condensates. Under conditions of vortex
confinement, the core of a higher index vortex filament separates into a fine
structure described by the system of near-parallel index-1 filaments, whose
positions are described by the system \eqref{fp}. This asymptotic description
has been rigorously established in \cite{ContrerasJerrard16} in the stationary
case, and in the case of nontrivial time dependent evolution in
\cite{JerrardSmets16}. To our knowledge, the rigorous analytic justification
of \eqref{fp} as a model of vortex filaments for the Euler equations of fluid
dynamics is open.

In this paper we consider families of time periodic solutions that
additionally are periodic in the spatial variable $s$. Our main result is a
bifurcation theory for periodic standing waves of this $n$ vortex ensemble,
giving rise to families of solutions that bifurcate from $n$ exactly parallel
filaments that revolve around an axis, with positions determined by a central
configuration. For $n=2$ this addresses the case of any two near-parallel
filaments with the same circulation. When $n\geq3$ we consider arbitrary
central configurations of vortex filaments. In all cases we consider, the
amplitudes of the standing waves that are constructed are restricted to a
Cantor set, due to an accumulation of resonances along the bifurcating
solution branches.

\begin{figure}[ptb]
\centering
\par
\begin{pspicture}(-5,-2)(5,2)\SpecialCoor
\psline{->}(0,-1.5)(0,1.5)
\psellipticarc[linestyle=dashed,arrowscale=1.5]{->}(0,1.5)(.5,.2){-90}{270}
\rput[b](0,1.8){$\omega$}
\psline(-4,-1.5)(-4,1.5)
\psellipticarc[linestyle=dashed,arrowscale=1.5]{->}(-4,1.5)(.5,.2){-180}{90}
\psellipticarc[linestyle=dashed,arrowscale=1.5]{->}(-4,0)(.5,.2){0}{-90}
\psellipticarc[linestyle=dashed,arrowscale=1.5]{->}(-4,-1.5)(.5,.2){-180}{90}
\psecurve(-3.5,3)(-4.5,1.5)(-3.5,0)(-4.5,-1.5)(-3.5,-3)
\psline(4,-1.5)(4,1.5)
\psellipticarc[linestyle=dashed,arrowscale=1.5]{->}(4,1.5)(.5,.2){0}{-90}
\psellipticarc[linestyle=dashed,arrowscale=1.5]{->}(4,0)(.5,.2){180}{90}
\psellipticarc[linestyle=dashed,arrowscale=1.5]{->}(4,-1.5)(.5,.2){0}{-90}
\psecurve(3.5,3)(4.5,1.5)(3.5,0)(4.5,-1.5)(3.5,-3)
\NormalCoor\end{pspicture}
\caption{Two vortex filaments that have large separation, that rotate
uniformly around the center with small frequency $\omega\sim0$. The
perturbation from straight filaments oscillates approximately as $u(s,t) \sim
r(e^{-it})\cos s$.}%
\end{figure}
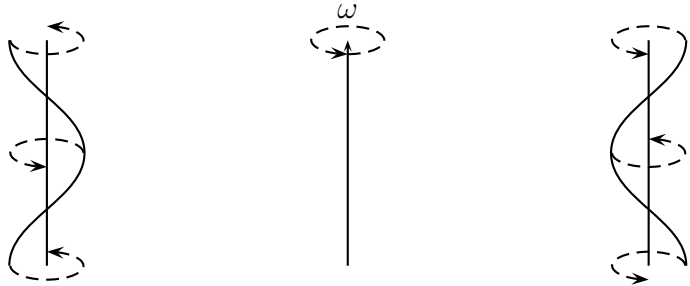

Specifically, as a consequence of the conclusions of our
Theorem~\ref{Theorem2} below, the following result holds for the $n$ vortex
filament problem. \newpage

\begin{theorem}
Let $a_{j}$ be a central configuration, satisfying
\[
\omega a_{j}=\sum_{i=1(i\neq j)}^{n}\frac{a_{j}-a_{i}}{\left\vert a_{j}
-a_{i}\right\vert ^{2}}\text{.}%
\]
Then the $n$ vortex filament problem (\ref{fp}) has solutions of the form%
\[
u_{j}(t,s)=a_{j}e^{it\omega}\left(  1+u(\Omega(r)t,s;r)\right)  \text{,}%
\]
where $\omega$ is a diophantine frequency, $r$ is a small amplitude varying
over a Cantor set, and $\Omega(r)=\Omega_{0}+\mathcal{O}(r^{2})$ with
$\Omega_{0}=\sqrt{1+2\omega}$ and
\[
u(t,s)=r\cos s\left(  \cos t-i\Omega_{0}\sin t\right)  +\mathcal{O}(r^{2})~.
\]

\end{theorem}

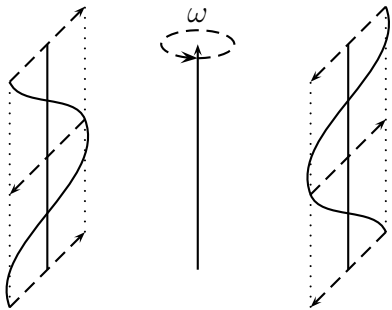
\begin{figure}[ptb]
\centering
\par
\begin{pspicture}(-3,-2)(3,2)
\SpecialCoor
\psline{->}(0,-1.5)(0,1.5)
\psellipticarc[linestyle=dashed,arrowscale=1.5]{->}(0,1.5)(.5,.2){-90}{270}
\rput[b](0,1.8){$\omega$}
\psline(-2,-1.5)(-2,1.5)
\psline[linestyle=dotted](-2.5,1)(-2.5,-2)
\psline[linestyle=dotted](-1.5,2)(-1.5,-1)
\psline[linestyle=dashed]{->}(-2.5,1)(-1.5,2)
\psline[linestyle=dashed]{->}(-1.5,.5)(-2.5,-.5)
\psline[linestyle=dashed]{->}(-2.5,-2)(-1.5,-1)
\psecurve(-1.5,3.5)(-2.5,1)(-1.5,.5)(-2.5,-2)(-1.5,-2.5)
\psline(2,-1.5)(2,1.5)
\psline[linestyle=dotted](1.5,1)(1.5,-2)
\psline[linestyle=dotted](2.5,2)(2.5,-1)
\psline[linestyle=dashed]{<-}(1.5,1)(2.5,2)
\psline[linestyle=dashed]{<-}(2.5,.5)(1.5,-.5)
\psline[linestyle=dashed]{<-}(1.5,-2)(2.5,-1)
\psecurve(1.5,2.5)(2.5,2)(1.5,-.5)(2.5,-1)(1.5,-3.5)
\NormalCoor
\end{pspicture}
\caption{Two vortex filaments with small separation, rotating with frequency
$\omega\sim\infty$. The perturbation from straight filaments oscillates
approximately as $u(s,t) \sim r(i\sqrt{2\omega}\sin\sqrt{2\omega}t)\cos s$.}%
\end{figure}

Many PDE's that describe physical phenomena exhibit the structure of an
infinite dimensional Hamiltonian system. Equations that model nonlinear wave
phenomena are such, and they typically possess equilibria that are elliptic
points in the sense of dynamical systems, for which one anticipates families
of nearby periodic and quasi-periodic solutions. However for Hamiltonian
PDE's, even the construction of periodic solutions often presents a small
divisor problem, due to the infinite number of degrees of freedom and the
spectral properties of the relevant linearized operators. Thus such
constructions are analogous to the analytic challenge of constructing
invariant tori in KAM theory.

In this paper the methods that are used to construct periodic solutions
involve a Lyapunov-Schmidt reduction along with a Nash-Moser procedure, a
strategy introduced in Craig and Wayne~\cite{CrWa93}. This technique is useful
in problems with multiple resonances, as it has the advantage that it does not
require the third Melnikov condition that is a feature in the classical KAM
methods. This procedure was generalized by Bourgain to construct quasiperiodic
solutions for Hamiltonian PDEs, in for instance~\cite{Bo95}, and was further
used an improved by Berti and Bolle~\cite{BeBo13} and a number of other authors.

The strategy of Lyapunov-Schmidt reduction is used to solve the range equation
through a Nash-Moser procedure, with smoothing operators given by Fourier
truncations of finite but increasing dimension. The convergence of the
procedure relies on estimates of the inverse of a sequence of finite
dimensional matrices of asymptotically large dimension. Approximate inversion
of linearized operators is the key to most applications of KAM theory to
Hamiltonian PDEs. Resolvent estimates that resemble Fr\"{o}hlich--Spencer
estimates are used to control the inverse of the linearization, projected off
of the kernel. However restrictions in the local coordinates of the kernel and
the frequency parameters have to be imposed in order to obtain these
estimates. These in turn require the excision of certain near resonant subsets
of parameter space in an inductive procedure. The estimates for the projected
inverse operator diverge, due to the small divisors, a divergence which is
overcome by the rapid convergence of the Newton method. However, at the end of
the process the set of parameters for which the scheme converges is reduced to
a Cantor set of positive measure, and the bifurcation equation is solved only
over this Cantor set.

In this paper, the bifurcation equation is analyzed with the use of the
symmetries. These symmetries imply the existence of two bifurcation branches,
one consisting of standing waves and one of traveling waves. The branch of
standing waves corresponds to a symmetry-breaking phenomenon from a
one-dimensional orbit to a three-dimensional orbit. Estimates on the measure
of the intersection of the branch of nontrivial standing waves with the Cantor
set finishes the proof. The branch of traveling wave solutions does not
involve small divisors; it is discussed in the Appendix along with other solutions.

The Nash-Moser procedure has been previously applied to nonlinear wave
equations and NLS equations in \cite{CrWa93}, \cite{CrWa94}, \cite{Bo95},
\cite{BeBo13} and \cite{Be14}, and for standing water waves in \cite{PlTo01}.
We have two purposes in developing the analysis of Hamiltonian PDEs and small
divisor problems in this paper. Firstly, the underlying problem of vortex
filament dynamics has a relevance to fluid dynamics, and in particular to
questions of Euler flows. Moreover, we have a secondary objective, which is to
present a simplified and more straightforward proof of the relevant estimates
of the inverse of the linearized operator, being a variant and a conceptually
simpler version of the classical Fr\"{o}hlich-Spencer estimates. We believe
that these will be useful for other small divisor problems, whether for finite
dimensional Hamiltonian systems in resonant situations for which the third
Melnikov condition is not expected to hold, or for problems of invariant tori
for Hamiltonian PDEs.

In section 1, we set up the question of existence of standing waves as a
problem of bifurcation theory in a space with symmetries, related to PDE
analogs of \cite{GoSc86} and \cite{IzVi03}. In section 2 we discuss the
Lyapunov-Schmidt reduction with the use of the Nash-Moser procedure. In
section 3 we present the estimates for the projected inverses, assuming
certain hypotheses on the separation of singular sites and estimates of the
spectra of the relevant linear operators. This step is at the heart of the
Nash-Moser procedure. In section 4 we estimate the measure of the set of
parameters for which the above hypotheses of separation and inversion are
satisfied. Finally in section 5 we show that solutions of the bifurcation
equation intersect the above set of good parameters, and show that the
intersection has asymptotically full measure. In the Appendix we discuss the
symmetries of the standing waves, and in addition we give a result about the
global bifurcation of periodic traveling waves, which has a much different
character than that of the standing waves.

\section{Central configurations of vortex filaments}

A central configuration of exactly parallel vortex filaments is a family of
straight and exactly parallel lines $(u_{j}(t),s)$, whose dynamics satisfy
$u_{j}(t)=e^{i\omega t}a_{j}$, where the coordinates $a_{j}\in\mathbb{C}$
satisfy
\begin{equation}
\omega a_{j}=\sum_{i=1,i\not =j}^{n}\frac{a_{j}-a_{i}}{|a_{j}-a_{i}|^{2}}
\label{Eqn:CentralConfiguration}%
\end{equation}
for all $j=1,\dots n$. Such configurations arise in studies of the $n$-body
problem, for example \cite{GaIz11}, \cite{MeHa91}, \cite{Ne01} and the
references contained in these papers.

Homographic solutions for the vortex filament problem are solutions with
\[
u_{j}(t,s)=w(t,s)a_{j},
\]
where $a_{j}$'s are complex numbers satisfying
\eqref{Eqn:CentralConfiguration}. In this class of solutions the shape of the
intersections of the filaments with a horizontal complex plane is homographic
with the shape of their intersection with any other horizontal plane
$\{x_{3}=c\}$ for any $x_{3}$ and at any time $t$. Homographic solutions of
this form satisfy the equations (\ref{Eqn:CentralConfiguration}) and
\begin{equation}
\partial_{t}w = i\left(  \partial_{ss}w + \omega\left\vert w\right\vert
^{-2}w\right)  \text{.} \label{pde}%
\end{equation}
The set of solutions $w(t,s)a_{j}$ foliate an invariant manifold. By rescaling
one may set $\omega=1$ in (\ref{Eqn:CentralConfiguration}). In the case $n=2$
of two filaments, the complement of this subspace, that is in center of
circulation coordinates, the orbit space is foliated by solutions of the
linear Schr\"{o}dinger equation, see \cite{KPV03}.

Solutions of equation \eqref{Eqn:CentralConfiguration} are known as central
configurations of the $n$-vortex problem, and solutions of this form have been
well studied in the literature. For instance, in \cite{KPV03} and
\cite{GaIz12}, a polygonal central configuration with a central filament is
discussed. Solutions to equation \eqref{pde} also generate homographic
solutions when a central filament with different circulation is fixed at the
central axis, see \cite{KPV03}.

Equation \eqref{pde} can be formulated as a Hamiltonian system, given by
\[
\partial_{t}w=i\partial_{\bar{w}}{H}(w)
\]
where the Hamiltonian for the system is
\[
{H} = \int_{0}^{2\pi} |\partial_{s}w|^{2}-\ln(|w|^{2}) \, ds~\text{.}%
\]
Since the Hamiltonian ${H}(w)$ is autonomous and invariant under change of
phase and translation, the energy ${H}$, the angular momentum
\[
{I} = \int_{0}^{2\pi}\left\vert w\right\vert ^{2}\,ds~,
\]
and the momentum
\[
{W}=\int_{0}^{2\pi} \bar{w}(i\partial_{s} w) \, ds
\]
are conserved quantities.

The simplest solutions to equation \eqref{pde} are relative equilibria of the
form $w(t,s)=e^{i\omega t}v(s)$. In these solutions the filaments turn around
a central axis at a constant uniform speed. For this class of solutions
equation \eqref{pde} becomes
\begin{equation}
\omega v=\partial_{ss}v+\left\vert v\right\vert ^{-2}v~, \label{ode}%
\end{equation}
with $\omega$ being the angular frequency.

One simple family of solutions of \eqref{ode} in helical form is given by
$v(s)=ae^{i\sigma s}$ with $\omega=-\sigma^{2}+a^{-2}$. Thus a continuum of
solutions of equation \eqref{pde} is given by
\begin{equation}
w(t,s)=ae^{i(\omega t+\sigma s)}\,\text{ with }\,\omega=-\sigma^{2}+a^{-2}
\label{sol}%
\end{equation}
which is parametrized by the vortex filament separation $a$. The solutions in
this continuum have a one dimensional orbit $ae^{i\theta}e^{i(\omega t+\sigma
s)}$ for $\theta\in{\mathbb{S}}^{1}$. Equation \eqref{pde} is similar to the
Kepler problem; other solutions to this equation are discussed in the appendix.

The equation \eqref{pde} is invariant under the Galilean transformation
\begin{equation}
e^{-i\alpha^{2}t}e^{i\alpha s}w(t,s-2\alpha t).
\end{equation}
Under the Galilean transformation, the solution \eqref{sol} generates the
family
\[
ae^{i((\omega-2\alpha\sigma-\alpha^{2})t+(\sigma+\alpha)s)}\,\text{.}%
\]
Since $\omega=-\sigma^{2}+a^{-2}$, by the choice $\alpha=-\sigma$, the
solution \eqref{sol} becomes $ae^{i\omega}$ with $\omega=a^{-2}$. Hence, under
the symmetry exhibited by the Galilean transformations, the different branches
$ae^{i(\omega t+\sigma s)}$ are transformed to the branch with $\sigma=0$.

From the previous remark, we may assume without loss of generality that the
bifurcation branch of periodic solutions is close to the co-rotating exactly
parallel solutions $ae^{i\omega t}$, with $\omega=1/a^{2}$. Subsequently,
using the Galilean transformation, any periodic solution near $ae^{i\omega t}$
can be reproduced as a solution that is a perturbation of $ae^{i(\omega
t+\sigma s)}$ for an arbitrary choice of $\sigma\in\mathbb{R}$.

The equation \eqref{pde} is invariant under the scaling
\[
\tau^{-1}w(\tau^{2}t,\tau s)\,\text{,}%
\]
so that any $P$-periodic boundary condition in $s$ may be fixed to $P=2\pi$.
However, once the spatial period has been fixed, the amplitude $a$ cannot be
scaled further. It therefore suffices to consider the problem of bifurcation
of solutions which have spatial period $2\pi$. The amplitude parameter $a$ has
the role of an external parameter in the problem, and solutions $ae^{i\omega
t}$ have different properties depending on $a$, due to patterns of resonance.

Finally, we address the bifurcation of $2\pi/\Omega$-periodic solutions in
time which are close to the exactly parallel central configuration $e^{i\omega
t}a$, using the change of coordinates
\begin{equation}
w(t,s)=ae^{\omega it}v(\Omega t,s)~\text{.} \label{cov}%
\end{equation}
Under rescaling the time variable, the $2\pi/\Omega$-periodic solutions of the
equation \eqref{pde} satisfy the equation
\begin{equation}
i\Omega\partial_{t}v=-\partial_{ss}v+\omega(1-|v|^{-2})v~\text{.} \label{pdev}%
\end{equation}
Since the equation depends of the amplitude $a$ only through the angular
frequency $\omega$, in the subsequent analysis we choose to parametrize
solution families of \eqref{pdev} through the frequency $\omega$.

\section{The Nash-Moser method}

By arguments of generic bifurcation, following the remarks in the appendix, it
is expected that there is one bifurcating branch of traveling waves and one of
standing waves. The traveling wave do not present a small divisor problem, and
we give a proof of their existence in the appendix.

To prove the existence of families of standing wave solutions, we use a
Nash-Moser method in a subspace of symmetries (\ref{sym}), a reduction which
simplifies several aspects of the proof; in particular in this setting
singular regions consist of isolated sites, the linearization is not
degenerate, and the kernel of the linear operator is one dimensional.

Using the change of variables $v=1+u$ in equation (\ref{pdev}), where $u$ is a
small perturbation, the bifurcation of periodic solutions are zeros of the
map
\begin{equation}
f(u;\Omega)=-i\Omega u_{t}-u_{ss}+\omega(u+\bar{u}) +\omega g(u,\overline
{u})\text{,} \label{f}%
\end{equation}
where $\Omega$ is the temporal frequency, and the nonlinearity is given by
\[
g(u,\bar{u})=\sum_{n=2}^{\infty}(-1)^{n}\bar{u}^{n} =\frac{\bar{u}^{2}}%
{1+\bar{u}}~.
\]
Our main result is stated in the following theorem.

\begin{theorem}
\label{Theorem2} For $\omega$ diophantine, with only one exceptional value
$\omega= \omega_{0}$, there exists $r_{0}>0$ and a Cantor subset
$\mathcal{C}\subset\lbrack0,r_{0}]$ with measure $\left\vert \mathcal{C}%
\right\vert \geq r_{0}(1-r_{0}C_{\beta})$ ($r_{0}\ll1$), such that the
operator $f(u;\Omega)$ has a nontrivial bifurcation branch of solutions of
$f(u;\Omega)=0$ parametrized by $r\in\mathcal{C}$. The solutions
$(u(s,t;r),\Omega(r))$ are analytic and periodic in $s$ and $t$, and they have
the following form;%
\[
u(t,s;r)=r\cos s\left(  \cos t-i\Omega_{0}\sin t\right)  +\mathcal{O}
(r^{2})\text{,}%
\]
where $\Omega(r)=\Omega_{0}+\Omega_{2}r^{2}+\mathcal{O}(r^{3})$ with
$\Omega_{0}=\sqrt{1+2\omega}$ and $\Omega_{2}\neq0$. Moreover, the solutions
$u(s,t)$ are Whitney smooth in $r$ and exhibit the symmetries
\begin{equation}
u(t,s)\ =u(t,-s)=\bar{u}(-t,s)\text{.} \label{sym}%
\end{equation}

\end{theorem}

In fact, solutions of this theorem satisfy the additional symmetry%
\[
u(t,s)=u(t+\pi,s+\pi)\text{.}%
\]
The solutions of Theorem \ref{Theorem2} have an orbit which is a $3$-torus,
given by $e^{i\theta}u_{j}(t+\varphi,s+\psi)$ for $(\theta,\varphi,\psi
)\in\mathbb{T}^{3}$.

The exceptional value $\omega_{0}$ of the rotational frequency parameter is
due to the failure of the nondegeneracy condition of the amplitude - frequency
map in the construction of the bifurcation branch of periodic standing wave solutions.

Theorem~\ref{Theorem2} follows the basic structure of the Lyapunov center
theorem, however it concerns the time evolution of a PDE and subsequently the
analysis must deal with the problem of small divisors. In addition the problem
naturally encounters a resonant situation corresponding to its translation
invariance in the variable $s\in\mathbb{R}$, which presents a second
difficulty. To overcome the latter, we seek solutions which are symmetric
under reflections in space and time. A posteriori we show that this set of
time periodic solutions are the unique standing wave solutions. In order to
work in the setting of symmetric solutions, we define the Hilbert space%
\[
L_{sym}^{2}({\mathbb{T}}^{2};\mathbb{C})= \{u\in L^{2}({\mathbb{T}}
^{2};\mathbb{C}):u(t,s)=u(t,-s)=\bar{u}(-t,s)\}\text{,}%
\]
with the inner product%
\[
\left\langle u_{1},u_{2}\right\rangle =\frac{1}{(2\pi)^{2}}\int_{{\mathbb{T}
}^{2}}u_{1}\bar{u}_{2}dtds\text{.}%
\]
A function $u\in L_{sym}^{2}$ may be written in a Fourier basis
\[
u=\sum_{k\in\mathbb{N}}a_{0,k}\cos ks+\sum_{j\in\mathbb{N}^{+},k\in\mathbb{N}%
}\left(  a_{j,k}\cos jt+ib_{j,k}\sin jt\right)  \cos ks
\]
with $a_{j,k},b_{j,k}\in\mathbb{R}$. We are using the notation that
$\mathbb{N}=\{0,1,..\mathbb{\}}$ and $\mathbb{N}^{+}=\{1,2,...\}$. The
linearization of the map $f$ about $u=0$ is%
\[
L(\Omega)u:=f^{\prime}(0;\Omega)u=-i\Omega u_{t}-u_{ss} + \omega(u+\bar
{u})\text{.}%
\]
Explicitly, $L(\Omega)$ is diagonal in the above basis, and the Fourier
component $j=0$ is
\[
L(\Omega)(a_{0,k}\cos ks)=(k^{2}+2\omega)a_{0,k}\cos ks ~.
\]
For the Fourier component $j\in\mathbb{N}^{+}$, the linear map is
\[
L(\Omega)\left(  \left(
\begin{array}
[c]{c}%
a_{j,k}\\
b_{j,k}%
\end{array}
\right)  \cdot\left(
\begin{array}
[c]{c}%
\cos jt\\
i\sin jt
\end{array}
\right)  \cos ks\right)  =M_{j,k}\left(
\begin{array}
[c]{c}%
a_{j,k}\\
b_{j,k}%
\end{array}
\right)  \cdot\left(
\begin{array}
[c]{c}%
\cos jt\\
i\sin jt
\end{array}
\right)  \cos ks\text{,}%
\]
where $M_{j,k}$ is the matrix
\[
M_{j,k}(\Omega)=\left(
\begin{array}
[c]{cc}%
k^{2}+2\omega & -\Omega j\\
-\Omega j & k^{2}%
\end{array}
\right)  .
\]

The matrix $M_{j,k}$ has eigenvalues
\[
\lambda_{j,k,\pm1}=k^{2}+\omega\pm\sqrt{j^{2}\Omega^{2}+\omega^{2}}\text{,}%
\]
and normalized eigenvectors
\[
v_{j,k,\pm1}=\frac{1}{c_{j,\pm1}}\left(
\begin{array}
[c]{c}%
\omega\pm\sqrt{j^{2}\Omega^{2}+\omega^{2}}\\
-j\Omega
\end{array}
\right)  \text{,}%
\]
where%
\[
c_{j,\pm1}=\sqrt{2}\left(  \omega^{2}+j^{2}\Omega^{2}\pm\omega\sqrt
{j^{2}\Omega^{2}+\omega^{2}}\right)  ^{1/2}\text{.}%
\]

The orthonormal eigenbasis is given by $e_{0,0,1}=1$, $e_{0,k,1}=\sqrt{2}\cos
ks$ and%
\[
e_{j,k,l}=c_{k}v_{j,k,l}\cdot\left(
\begin{array}
[c]{c}%
\cos jt\\
i\sin jt
\end{array}
\right)  \cos ks
\]
for $(j,k,l)\in\mathbb{N}^{+}\mathbb{\times N\times}\{1,-1\}$, where
$c_{0}=\sqrt{2}$ and $c_{k}=2$ ($k\neq0$). That is, we have $\left\langle
e_{x},e_{y}\right\rangle =\delta_{x,y}$ and for every function $u\in
L_{sym}^{2}$,%
\[
u=\sum_{x\in\Lambda}\left\langle u,e_{x}\right\rangle e_{x}\text{,\qquad}
L(\Omega)u=\sum_{x\in\Lambda}\lambda_{x}\left\langle u,e_{x}\right\rangle
e_{x}\text{,}%
\]
where
\[
x\in\Lambda=\mathbb{N}^{+}\mathbb{\times N\times}\{1,-1\}\cup
\{0\}\mathbb{\times N\times}\{1\}\text{.}%
\]

Define the analytic norm%
\[
\left\Vert u\right\Vert _{\sigma}^{2}=\sum_{(j,k,l)\in\Lambda}\left\langle
u,e_{j,k,l}\right\rangle ^{2}e^{2\left\vert (j,k)\right\vert \sigma
}\left\langle (j,k)\right\rangle ^{2s}\text{,}%
\]
where $\left\langle (j,k)\right\rangle =\sqrt{1+j^{2}+k^{2}}$. Even though the
eigenfunctions $e_{x}$ depend on the frequency parameter $\Omega$, and the
norm $\| u \|_{\sigma}$ is defined through this system of eigenfunctions, the
norm is in fact independent of $\Omega$ since the basis $\{e_{x}%
\}_{x\in\Lambda}$ is orthonormal.

\begin{lemma}
The Hilbert space $h_{\sigma}$ of analytic functions given by%
\[
h_{\sigma}=\{u\in L_{sym}^{2}({\mathbb{T}}^{2};\mathbb{C}):\left\Vert
u\right\Vert _{\sigma}<\infty\}~.
\]
is an algebra for $s>1$;%
\[
\left\Vert uv\right\Vert _{\sigma}\leq c_{\sigma,s}\left\Vert u\right\Vert
_{\sigma}\left\Vert v\right\Vert _{\sigma}\text{.}%
\]

\end{lemma}

\begin{proof}
The result is quite standard. Let $u\in L^{2}({\mathbb{T}}^{2},\mathbb{C)}$,
then the norm%
\begin{equation}
\sum_{(j,k)\in\mathbb{Z}^{2}}\left\vert \left\langle u,e^{i(jt+ks)}%
\right\rangle \right\vert ^{2}e^{2\left\vert (j,k)\right\vert \sigma
}\left\langle (j,k)\right\rangle ^{2s} \label{n}%
\end{equation}
has the algebra property under the convolution for $s>1$, see \cite{CrWa94}.
The result follows from the fact that the two norms $\left\Vert \cdot
\right\Vert _{\sigma}$ and (\ref{n}) are equivalent in $L_{sym}^{2}$.
\end{proof}

\begin{lemma}
The nonlinear operator $g(u,\bar{u})\ $satisfies
\[
\left\Vert g(u,\bar{u})\right\Vert _{\sigma}<c_{\sigma}\left\Vert u\right\Vert
_{\sigma}^{2}~,
\]
for small $\left\Vert u\right\Vert _{\sigma}$. Thus, the map $f$ is well
defined and continuous in
\[
\left\{  u\in h_{\sigma}:\left\Vert u\right\Vert _{\sigma}<c_{\sigma}%
^{-1}\right\}  ~.
\]

\end{lemma}

\begin{proof}
That the map is well defined follows from the equivariant property and the
first statement. The first inequality follows from the algebraic property,%
\[
\left\Vert g(u,\bar{u})\right\Vert _{\sigma}=\left\Vert \sum_{n=2}^{\infty
}(-1)^{n}\bar{u}^{n}\right\Vert _{\sigma}\leq\sum_{n=2}^{\infty}c_{\sigma
,s}^{n-1}\left\Vert u\right\Vert _{\sigma}^{n}=\frac{c_{\sigma,s}\left\Vert
u\right\Vert _{\sigma}^{2}}{1-c_{\sigma,s}\left\Vert u\right\Vert _{\sigma}%
}<c_{\sigma}\left\Vert u\right\Vert _{\sigma}^{2}~.
\]

\end{proof}

The eigenvalue $\lambda_{j,k,-1}(\Omega)$ is zero when the frequency is equal
to%
\[
\Omega_{j,k}=j^{-1}\sqrt{k^{4}+2k^{2}\omega}\text{.}%
\]
In the case that there are no additional resonances, and without loss of
generality, we may analyze the branch of bifurcation from the eigenvalue
$\lambda_{1,1,-1}$; all other being equivalent up to scaling (see the
appendix). In the next proposition we prove that there are no additional
resonances if $\omega$ is irrational.

\begin{proposition}
Let $\Omega_{0}=\sqrt{1+2\omega}$. For $\omega$ irrational, the kernel of the
map $L(\Omega_{0})$ has dimension one, corresponding to the eigenfunction
$e_{1,1,-1}$.
\end{proposition}

\begin{proof}
The eigenvalues of $M_{j,k}(\Omega_{0})$ are $\lambda_{j,k,l}(\Omega_{0})$ for
$l=\pm1$. The determinant of the matrix $M_{j,k}(\Omega_{0})$ is
\[
d_{j,k}(\Omega_{0})=\lambda_{j,k,1}\lambda_{j,k,-1} = 2\left(  k^{2} -
j^{2}\right)  \omega+ \left(  k^{4}-j^{2}\right)  \text{.}%
\]
Since the frequency $\Omega_{0}$ is chosen such that $M_{1,1}$ is not
invertible, then $d_{1,1}(\Omega_{0})=0$. Since $\omega$ is irrational, then
the determinant $d_{j,k}$ is zero only if both numbers $k^{2}-j^{2}$ and
$k^{4}-j^{2}$ are zero. For $(j,k,l)\in\Lambda$, the condition happens only
when $(j,k)=(1,1)$ or $(j,k)=(0,0)$. But the eigenvalue $\lambda
_{0,0,1}=2\omega$ is always positive, then the only zero eigenvalue in the
lattice $\Lambda$ is $\lambda_{1,1,-1}(\Omega_{0})$.
\end{proof}

The eigenvalue $\lambda_{0,0,-1}$ is always zero, and on $L^{2}({\mathbb{T}%
}^{2};\mathbb{C})$ it contributes to the null space of $L(\Omega)$. However on
the space $L_{sym}^{2}$ that takes the symmetries into account, the $(0,0,-1)$
is not an element of the lattice $\Lambda$, and therefore $e_{0,0,-1}\notin
L_{sym}^{2}$ does not contribute to the respective null space.

For $A\subseteq\Lambda$ a subset of lattice points, define $P_{A}$ as the
projection onto the Fourier components $x\in A\subseteq\Lambda$,%
\[
P_{A}u = \sum_{x\in A}\left\langle u,e_{x}\right\rangle e_{x}~\text{.}%
\]
The null space of $L(\Omega_{0})$ is supported on only one lattice point $N :=
\{x = (1,1,-1)\in\Lambda\}$ because of the symmetry reduction; hence
$L(\Omega_{0})e_{x}=0$. Thus, $P_{N}$ is the projection onto the one
dimensional kernel, and $P_{\Lambda\backslash N}$ is the projection onto its
complement. In future studies of quasi-periodic solutions the set $N$ will in
general comprise a possibly large but finite number of lattice sites.

\subsection*{Lyapunov-Schmidt reduction}

Define the decomposition $u=v+w$ into its components in the kernel and the
range of $L(\Omega_{0})$;%
\[
v = P_{N}u,\qquad w=P_{\Lambda\backslash N}u ~.
\]
Then the zeros of the nonlinear operator $f(u;\Omega)$ defined in \eqref{f}
are the solutions of the pair of equations
\[
P_{N}f(v+w;\Omega) = 0 ~,\qquad P_{\Lambda\backslash N}f(v+w;\Omega)=0 ~.
\]
Let $r\in\mathbb{R}$ be a parametrization of the kernel of $f^{\prime
}(0;\Omega_{0})$ given by%
\[
v(r)=re_{1,1,-1}\in L_{sym}^{2}({\mathbb{T}}^{2};\mathbb{C}) \text{.}%
\]
The strategy of the construction consists in solving $w(r,\Omega)$ in the
range equation $P_{\Lambda\backslash N}f(v(r)+w;\Omega)=0$, using a Nash-Moser
procedure. The key aspect of this step is that it is a small divisor problem,
which implies that the function $w(r,\Omega)$ can only be constructed in a
Cantor set of parameters $(r,\Omega)$ in a neighborhood of $(0,\Omega_{0})$.

\subsection*{Initial Nash-Moser step}

A first approximation of the solution is constructed projecting in the ball
$B_{0}$ of radius $L_{0}$,%
\[
B_{0}=\{(j,k)\in\Lambda\backslash N:\left\vert j\right\vert +\left\vert
k\right\vert < L_{0}\}\text{.}%
\]
To solve the first step of the Nash-Moser iteration, we use the following proposition;

\begin{lemma}
For a diophantine number $\omega$, if $\Omega$ is such that $\left\vert
\Omega-\Omega_{0}\right\vert <c\gamma/L_{0}^{2\tau+1}$, then
\[
\left\vert \lambda_{j,k,-1}\left(  \Omega\right)  \right\vert > c\gamma
/L_{0}^{2\tau+2}%
\]
for all $(j,k)$ with $\left\vert \left(  j,k\right)  \right\vert <L_{0}$.
\end{lemma}

\begin{proof}
Using the determinant of $M_{j,k}(\Omega_{0})$, and the fact that $\left\vert
q\omega- p\right\vert \geq\gamma/\left\vert q\right\vert ^{\tau}$, we have
\[
\left\vert d_{j,k}(\Omega_{0})\right\vert \geq\gamma/\left\vert 2\left(
k^{2}-j^{2}\right)  \right\vert ^{\tau}\geq c\gamma/L_{0}^{2\tau}\text{.}%
\]
Since $\left\vert \lambda_{j,k,1}\right\vert \leq cL_{0}^{2}$, then%
\[
\left\vert \lambda_{j,k,-1}(\Omega_{0})\right\vert \geq c\gamma/L_{0}
^{2\tau+2} ~.
\]
The result follows from the inequalities%
\begin{align*}
\left\vert \lambda_{j,k,-1}\left(  \Omega\right)  )\right\vert  &
\geq\left\vert \lambda_{j,k,-1}(\Omega_{0})\right\vert -\left\vert
\lambda_{j,k,-1}\left(  \Omega\right)  -\lambda_{j,k,-1}(\Omega_{0}
)\right\vert \\
&  \gtrsim\gamma/L_{0}^{2\tau+2}-\left\vert j\right\vert \left\vert
\Omega-\Omega_{0}\right\vert \gtrsim\gamma/L_{0}^{2\tau+2} ~.
\end{align*}

\end{proof}

By the implicit function theorem and the lemma above, there is a solution
$w_{0}(r,\Omega)$ of $P_{B_{0}}f(v(r)+w_{0},\Omega)=0$, which is defined over
the regime of parameters $(r,\Omega)$ with $0\leq r<r_{0}$ and $|\Omega
-\Omega_{0}|<c\gamma/L_{0}^{2\tau+1}$.

\subsection*{Subsequent Nash-Moser steps}

Let $P_{B_{n}}$ be the projection in the ball
\[
B_{n}=\{(j,k)\in\Lambda\backslash N:\left\vert j\right\vert +\left\vert
k\right\vert <L_{n}\}
\]
with $L_{n}=2^{n}L_{0}$ . The Nash-Moser method consists in providing a better
approximation $w_{n}$ at the $n^{th}$ step, where
\[
w_{n}(r,\Omega)=w_{n-1}(r,\Omega)+\delta w_{n}\text{.}%
\]
Here $\delta w_{n}$ is the correction given by the approximate Newton's
method
\begin{equation}
\label{Eqn:ApproxNewtonIteration}\delta w_{n} = -G_{B_{n}}P_{B_{n}}
f(w_{n-1}(r,\Omega);r,\Omega),
\end{equation}
where the inverse operator to $\partial_{w}f(w_{n-1}(r,\Omega);r,\Omega)$ on
$B_{n}$ is given by%
\[
G_{B_{n}}=\left(  P_{B_{n}}\partial_{w}f(w_{n-1})P_{B_{n}}\right)  ^{-1}
\text{.}%
\]

If the operator norm of the inverse has a tame estimate
\begin{equation}
\left\Vert G_{B_{n}}\right\Vert _{\sigma_{n}}\lesssim\gamma_{n}^{-2}\frac
{1}{d_{n}}\text{,}\label{green}%
\end{equation}
where $\sigma_{n}=\sigma_{n-1}-2\gamma_{n}$, $d_{n}<d_{n-1}$ and
$(r,\Omega)\in\mathcal{N}_{n}$, then we obtain an inductive estimate for
$\delta w_{n}$ of the form
\[
\left\Vert \delta w_{n}\right\Vert _{\sigma_{n}}<Cr^{2}e^{-\kappa^{n}%
}\ \ \text{with}\ \ \kappa\in(1,2)\text{.}%
\]
The following theorem quantifies the smallness condition that are required for
the convergence of the iterative scheme.

\begin{theorem}
If $d_{n}=L_{n}^{-\beta}$ for $\beta>3/2$ and $\sum_{n=1}^{\infty}2\gamma
_{n}\rightarrow\sigma_{0}/2$, from the estimate (\ref{green}), then
$w_{n}(r,\Omega)$ converges to the solution $w(r,\Omega)\in h_{\sigma_{0}/2}$
of
\[
P_{\Lambda\backslash N}f(v(r)+w(r,\Omega),\Omega)=0\text{,}%
\]
which is Whitney smooth in $(r,\Omega)$ over the Cantor set $\mathcal{N}%
=\cap_{n=1}^{\infty}\mathcal{N}_{n}$ (see the result in \cite{Cr00}).
\end{theorem}

\section{Estimates of the inverse}

Approximate inversion of the linearized operator $\partial_{w}f(v(r)+w;\Omega
)$ is key to applications of the Nash-Moser method. In this paper we follow an
approach that is motivated by techniques that have been developed to study
Anderson localization \label{FrSp}. We use the basic approach developed by
Craig and Wayne in \cite{CrWa93} with innovations by Bourgain \cite{Bo95}, and
Berti and Bolle \cite{BeBo13}, but we introduce certain simplifications that
clarify the method. We expect that this version will be useful in further
applications of Nash-Moser methods to Hamiltonian PDEs.

In this paper we describe the situation in which singular sites occur in
uniformly bounded clusters, and clusters are separated asymptotically in the
lattice. This is typically the case of time periodic solutions for PDEs in one
space dimension.

For the particular case of the problem of vortex filaments, we will be making
certain assumption on the nature and the spectra of the relevant linearized
operators. These assumptions will be verified in the subsequent Section 4,
where in order to do so we excise certain regions of parameter space
$(r,\Omega)\in\mathcal{N}$.

Given $w\in P_{\Lambda\backslash N}L_{sym}^{2}$, the projection $P_{\Lambda
\backslash N}\partial_{w}f$ of the linearization of the nonlinear operator
(\ref{f}) at the point $v(r)+w\in L_{sym}^{2}$ is given by
\begin{equation}
\label{Eqn:LinearizedP-equation}P_{\Lambda\backslash N}\partial_{w}%
f(v(r)+w;\Omega) = P_{\Lambda\backslash N} (L(\Omega)+\omega\partial
_{w}g(v+w))P_{\Lambda\backslash N}\text{.}%
\end{equation}
Because (\ref{fp}) is a Hamiltonian PDE, this linear operator is Hermitian.

In the Fourier basis $\{e_{x}\}$ the operator \eqref{Eqn:LinearizedP-equation}
is expressed by the matrix%
\[
H(w;r,\Omega)=D(\Omega)+T(w;r,\Omega)\text{,}%
\]
where, for $x,y\in\Lambda\backslash N$, the matrix $D$ is a diagonal;
\[
D(\Omega)=\hbox{\rm diag}_{x\in\Lambda\backslash N}(\lambda_{x}(\Omega
))\text{,}%
\]
and $T$ is a T\"{o}plitz linear operator
\[
T(w;r,\Omega)(x,y)=\left\langle \omega\partial_{w}g(v(r)+w)e_{y}%
,e_{x}\right\rangle \text{.}%
\]

Estimates of operators acting on analytic spaces $h_{\sigma}$ are realized
with the operator norm:
\[
\left\Vert G\right\Vert _{\sigma}:=\max\left\{
\begin{array}
[c]{c}%
\sup_{x}\sum_{y}\left\vert G(x,y)\right\vert e^{\sigma\left\vert
x-y\right\vert }\left\langle x-y\right\rangle ^{s}\\
\sup_{y}\sum_{x}\left\vert G(x,y)\right\vert e^{\sigma\left\vert
x-y\right\vert }\left\langle x-y\right\rangle ^{s}%
\end{array}
\right\}  .
\]
Consider the restriction of the linearized operator $H=P_{\Lambda\backslash
N}\partial_{w}f(v(r)+w;\Omega)$ to $\ell^{2}(E)$, where $E\subseteq
\Lambda\backslash N$ is a region of lattice sites of $\Lambda$. Its inverse
operator is denoted by%
\[
G_{E}=\left(  P_{E}HP_{E}\right)  ^{-1}%
\]

\begin{definition}
Let $A,B\subset\Lambda$. For the restriction of a matrix $H$, a linear mapping
of $\ell^{2}(\Lambda)$ to a mapping $\ell^{2}(B)\rightarrow\ell^{2}(A)$, we
define%
\[
H_{A}^{B}=P_{A}HP_{B}\text{.}%
\]
With this notation we have $(G_{E})_{A}^{B}=P_{A}(P_{E}HP_{E})^{-1}P_{B}$.
\end{definition}

Fix $w=w_{n-1}$ the approximation of the $(n-1)^{th}$ step of the Nash-Moser
iteration. The main linear estimate is of the matrix
\[
G_{B_{n}}(w_{n-1})=\left(  P_{B_{n}}H(w_{n-1})P_{B_{n}}\right)  ^{-1}\text{,}%
\]
where the region $B_{n}\subseteq\Lambda\backslash N$ is a ball centered at the
origin of large radius $L_{n}$.

\begin{definition}
We say that a site $x\in\Lambda$ is \emph{regular} if $\left\vert \lambda
_{x}\right\vert >d_{0}$. The subset of regular sites is denoted by
$A\subset\Lambda.$

We say that a site is \emph{singular} if $\left\vert \lambda_{x}\right\vert
\leq d_{0}$. We define $\mathcal{S}_{n}\subset\Lambda$ to be the set of
singular sites in the annulus $B_{n}\backslash B_{n-1}$.
\end{definition}

With these definitions we have the following decomposition of the lattice as a
disjoint union;
\[
\Lambda= A\cup_{n=1}^{\infty}\mathcal{S}_{n}\text{.}%
\]
The hypotheses we use to control the norm of $G_{B_{n}}$ are:

\begin{description}
\item[(h1)] The non-diagonal T\"{o}plitz matrix $T$ has the exponential decay
property%
\[
\left\Vert T(u)\right\Vert _{\sigma_{n}}\leq Cr_{0}%
\]
for $\left\Vert u\right\Vert _{\sigma_{n}}<r_{0}$.

\item[(h2)] The set of singular points in $B_{n}\backslash B_{n-1}$ is the
union of bounded regions $S_{j}$ that are pairwise separated. That is,
\[
\mathcal{S}_{n}=\cup_{j}S_{j}\text{,\qquad}\hbox{\rm rad}(S_{j})<c_{0}%
~\text{,}%
\]
and for $S_{j},S_{i}\subset\mathcal{S}_{n}\cup\mathcal{S}_{n-1}$, we assume
that%
\[
\hbox{\rm dist}(S_{i},S_{j})>4\ell_{n}~\text{.}%
\]

\item[(h3)] For $C(S_{j})$ the tubular neighborhood of radius $\ell_{n}$
around $S_{j}$, assume that the spectra of $H_{C(S_{j})}$ are bounded away
from zero by $d_{n}$. That is%
\[
\hbox{\rm spec}(H_{C(S_{j})})\subset\mathbb{R}\backslash\lbrack-d_{n}%
,d_{n}]~\text{.}%
\]

\end{description}

These will be shown to hold inductively for parameters $(r,\Omega)$ in the set
$\mathcal{N}_{n}$. We defined recursively
\[
\sigma_{n}=\sigma_{n-1}-2\gamma_{n}.
\]

\begin{theorem}
\label{Thm4}Assume hypothesis (h1)-(h3), and suppose that $cr_{0}/d_{0}<1/4$,
and
\[
c_{\gamma_{n}}r_{0}e^{-\gamma_{n}\ell_{n}}/d_{n}<1/4 ~,
\]
where
\[
c_{\gamma}=c\sum_{x\in\mathbb{Z}}e^{-\gamma\left\vert x\right\vert }
\lesssim\gamma^{-2}\text{.}%
\]
Let $E_{n}=B_{n-1}\cup\mathcal{S}_{n}\cup A_{n}$ with $A_{n}$ consisting of
regular sites in $\Lambda\backslash B_{n-1}$, then we have that
\begin{equation}
\left\Vert G_{E_{n}}(w_{n-1})\right\Vert _{\sigma_{n}}\lesssim c_{\gamma_{n}%
}/d_{n}\text{.}%
\end{equation}

\end{theorem}

The proof of Theorem \ref{Thm4} builds on a sequence of lemmas that serve to
estimate the norms of local Hamiltonians about singular sites.

Let $C_{n}(S_{j})$ be the neighborhood of radius $\ell_{n}$ around
$S_{j}\subset\mathcal{S}_{n}$,
\[
C_{n}(S_{j})=\{x\in\Lambda:\hbox{\rm dist}(x,S_{j})<\ell_{n}\}\text{.}%
\]
Because of hypothesis (h2) the sets $C_{n}(S_{j})$ are disjoint for distinct
$S_{j}$. We prove the following fact, omitting for clarity some instances of
the indices $n$ and $j$.

\begin{lemma}
\label{Lem1}Assume $cr_{0}/d_{0}<1/2$. From hypotheses (h1)-(h3), for
$\left\Vert u\right\Vert _{\sigma_{n}}<r_{0}$ and for any regular subset
$E\subset A$ or for $C(S)$ an $\ell$-tubular neighborhood of a singular region
$S$, we have
\[
\left\Vert G_{E}\right\Vert _{\sigma_{n}}\leq\frac{2}{d_{0}},\qquad\left\Vert
G_{C(S)}\right\Vert _{\sigma_{n}}\leq\frac{C}{d_{n}}\text{.}%
\]

\end{lemma}

\begin{proof}
Since $E$ is regular, $\left\Vert \left(  D_{E}^{E}\right)  ^{-1}T_{E}%
^{E}\right\Vert _{\sigma_{n}}\leq cr_{0}/d_{0}<1/2$, then%
\[
\left\Vert G_{E}\right\Vert _{\sigma_{n}}=\left\Vert \left(  I+\left(
D_{E}^{E}\right)  ^{-1}T_{E}^{E}\right)  ^{-1}\left(  D_{E}^{E}\right)
^{-1}\right\Vert _{\sigma_{n}}\leq\frac{2}{d_{0}}\text{.}%
\]

Since $H_{C(S)}$ is self-adjoint, by hypothesis (h3), the $L^{2}$-norm of
$G_{C(S)}$ is bounded by $1/d_{n}$, $\left\Vert G_{C(S)}\right\Vert _{0}%
\leq1/d_{n}$. Since the set $S$ has radius bounded by $c_{0}$, then%
\[
\left\Vert (G_{C(S)})_{S}^{S}\right\Vert _{\sigma_{n}}\leq e^{\sigma_{n}c_{0}%
}c_{0}^{s}\left\Vert (G_{C(S)})_{S}^{S}\right\Vert _{0}\leq\frac{C}{d_{n}%
}\text{.}%
\]
Since singular regions are separated, then the set $E=C(S)\setminus S$ is
regular. From the self-adjoint property, only two more cases require
verification,
\[
\left\Vert (G_{C(S)})_{E}^{S}\right\Vert _{\sigma_{n}}\leq C/d_{n}\text{ and
}\left\Vert (G_{C(S)})_{E}^{E}\right\Vert _{\sigma_{n}}\leq C/d_{n}.
\]

Let us define the connection matrix
\[
\Gamma:=H_{C(S)}-H_{E}\oplus H_{S}=T_{E}^{S}+T_{S}^{E}.
\]
From resolvent identities we have that
\[
G_{C(S)}=G_{E}\oplus G_{S}-G_{E}\oplus G_{S}\Gamma G_{C(S)}\text{,}%
\]
then
\begin{equation}
(G_{C(S)})_{E}^{S}=G_{E}T_{E}^{S}(G_{C(S)})_{S}^{S}\text{.}%
\end{equation}
Therefore, for the first case we have
\[
\left\Vert (G_{C(S)})_{E}^{S}\right\Vert _{\sigma_{n}}\leq C\frac{r_{0}}%
{d_{0}}\frac{1}{d_{n}}\text{.}%
\]

From resolvent identities, we have that
\[
G_{C(S)}=G_{E}\oplus G_{S}-G_{E}\oplus G_{S}\Gamma G_{E}\oplus G_{S}%
+G_{E}\oplus G_{S}\Gamma G_{C(S)}\Gamma G_{E}\oplus G_{S},
\]
then%
\begin{equation}
(G_{C(S)})_{E}^{E}=G_{E}+G_{E}T_{E}^{S}(G_{C(S)})_{S}^{S}T_{S}^{E}%
G_{E}\text{.}%
\end{equation}
Thus, for the second inequality we have
\[
\left\Vert (G_{C(S)})_{E}^{E}\right\Vert _{\sigma_{n}}\leq\frac{2}{d_{0}%
}+C\frac{r_{0}^{2}}{d_{0}^{2}}\frac{1}{d_{n}}\leq C\frac{1}{d_{n}}\text{.}%
\]

\end{proof}

\begin{lemma}
\label{Lem2}Let $P_{x}G$ be the projection on the $x^{th}$ row of $G$, then
for any $\sigma>\gamma>0$%
\[
\left\Vert G\right\Vert _{\sigma-\gamma}\leq c_{\gamma}\sup_{x}\left\Vert
P_{x}G\right\Vert _{\sigma}\text{.}%
\]
Let $A$ and $B$ be two sets such that $\hbox{\rm dist}(A,B)\geq\ell$, then%
\[
\left\Vert P_{A}GP_{B}\right\Vert _{\sigma-\gamma}\leq e^{-\gamma\ell
}\left\Vert G\right\Vert _{\sigma}\text{.}%
\]

\end{lemma}

\begin{proof}
The first result follows from the inequalities
\[
\sup_{y}\sum_{x}\left\vert G(x,y)\right\vert e^{(\sigma-\gamma)\left\vert
x-y\right\vert }\left\langle x-y\right\rangle ^{s}\leq\sup_{y}\sum
_{x}\left\Vert P_{x}G\right\Vert _{\sigma}e^{-\gamma\left\vert x-y\right\vert
}\leq c_{\gamma}\sup_{x}\left\Vert P_{x}G\right\Vert _{\sigma}~\text{,}%
\]
and
\[
\sup_{x}\sum_{y}\left\vert G(x,y)\right\vert e^{(\sigma-\gamma)\left\vert
x-y\right\vert }\left\langle x-y\right\rangle ^{s}\leq\sup_{x}\left\Vert
P_{x}G\right\Vert _{\sigma}~.
\]

The second result follows from the long step $\ell$ between $A$ and $B$ that
gives the estimate $e^{-\gamma\ell}$ in the exponential decay of the norm
$\sigma-\gamma$. That is, since $\left\vert x-y\right\vert >\ell$ for $x\in A$
and $y\in B$, then
\[
\sup_{y\in B}\sum_{x\in A}\left\vert G(x,y)\right\vert e^{(\sigma
-\gamma)\left\vert x-y\right\vert }\left\langle x-y\right\rangle ^{s}\leq
e^{-\gamma\ell}\left\Vert G\right\Vert _{\sigma}~\text{,}%
\]
and similarly for the supremum over $x$.
\end{proof}

\begin{figure}[th]
\centering
\par
\begin{pspicture}(-3,-3)(3,3)
\SpecialCoor
\psframe[linewidth=2pt](-2.5,2.5)(2.5,-2.5)
\psframe(-2.5,2.5)(-.5,1)
\psframe(2.5,-2.5)(.5,-1)
\psframe(-1,1)(1,-1)
\psline[linestyle=dotted](-1,2.5)(-1,-2.5)
\psline[linestyle=dotted](1,-2.5)(1,2.5)
\psline[linestyle=dotted](-2.5,1)(2.5,1)
\psline[linestyle=dotted](-2.5,-1)(2.5,-1)
\rput(1.5,1.5){$0$}
\rput(-1.5,-1.5){$0$}
\rput(0,0){$G_{A_n}$}
\rput(-1.7,1.7){$G_{C(B_{n-1})}$}
\rput(1.7,-1.7){$G_{C(S_{j})}$}
\rput[b](1.7,2.5){$S_{j}$}
\rput[b](-1.5,2.5){$C(B_{n-1})$}
\rput[b](0,2.5){$A_{n}$}
\rput[r](-2.5,-1.7){$S_{j}$}
\rput[r](-2.5,1.7){$B_{n-1}$}
\rput[r](-2.5,0){$A_{n}$}
\NormalCoor
\end{pspicture}
\caption{Preconditioner matrix $L_{n}$}%
\end{figure}
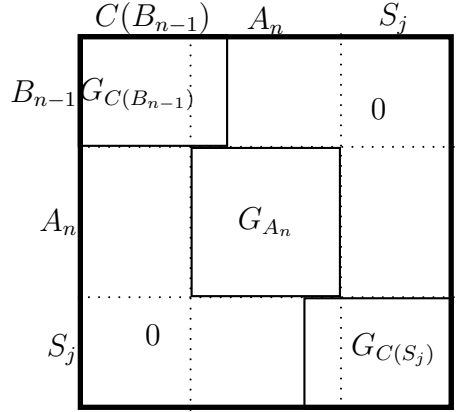

\begin{proof}
[Theorem \ref{Thm4}]The main estimate on the inverse
\[
G_{B_{n}}=\left(  P_{B_{n}}H(w_{n-1})P_{B_{n}}\right)  ^{-1}%
\]
is obtained by a approach distilled from \cite{BeBo13}; we construct a
preconditioner matrix $L_{n}$ from $G_{C_{n}(B_{n-1})}$, $G_{A_{n}}$ and
inverses $G_{C_{n}(S_{j})}$ of the local Hamiltonians $H_{C_{n}(S_{j})}$ for
$S_{j}\subset\mathcal{S}_{n}$. These inverses satisfy estimates as follows:
Since $A_{n}$ is regular, then
\begin{equation}
\left\Vert G_{A_{n}}\right\Vert _{\sigma_{n}}\lesssim1/d_{0}.
\end{equation}
By (h2) and Lemma \ref{Lem3}, then
\begin{equation}
\left\Vert G_{C_{n}(B_{n-1})}\right\Vert _{\sigma_{n-1}}\lesssim
c_{\gamma_{n-1}}/d_{n-1}. \label{ps}%
\end{equation}
By (h1)-(h3) and Lemma \ref{Lem1}, then%
\begin{equation}
\left\Vert G_{C_{n}(S_{j})}\right\Vert _{\sigma_{n-1}}\lesssim1/d_{n}.
\end{equation}

We define the preconditioner matrix $L_{n}$ as
\begin{equation}
L_{n}=G_{A_{n}}+P_{B_{n-1}}G_{C_{n}(B_{n-1})}+\sum_{S_{j}\subset
\mathcal{S}_{n}}P_{S_{j}}G_{C_{n}(S_{j})}\text{.} \label{Ln}%
\end{equation}
Thus, we have that
\begin{align*}
L_{n}(H_{E_{n}}^{E_{n}})  &  =G_{A_{n}}(H_{A_{n}}^{A_{n}}+H_{A_{n}}%
^{E_{n}\backslash A_{n}})+P_{B_{n-1}}G_{C_{n}(B_{n-1})}(H_{C_{n}(B_{n-1}%
)}^{C_{n}(B_{n-1})}+H_{C_{n}(B_{n-1})}^{E_{n}\backslash C_{n}(B_{n-1})})\\
&  +\sum_{S_{j}\subset\mathcal{S}_{n}}P_{S_{j}}G_{C_{n}(S_{j})}(H_{C_{n}%
(S_{j})}^{C_{n}(S_{j})}+H_{C_{n}(S_{j})}^{E_{n}\backslash C_{n}(S_{j}%
)})=I_{E_{n}}+K_{n},
\end{align*}
where, using that $G_{A_{n}}H^{A_{n}}_{A_{n}} = I_{A_{n}}$ etc., we have that
\[
K_{n}=G_{A_{n}}H_{A_{n}}^{E_{n}\backslash A_{n}}+P_{B_{n-1}}G_{C_{n}(B_{n-1}
)}H_{C_{n}(B_{n-1})}^{E_{n}\backslash C_{n}(B_{n-1})}+\sum_{S_{j}
\subset\mathcal{S}_{n}}P_{S_{j}}G_{C_{n}(S_{j})}H_{C_{n}(S_{j})}
^{E_{n}\backslash C_{n}(S_{j})}\text{.}%
\]
We conclude from (\ref{LnH}) that as long as $(I_{E_{n}}+K_{n})^{-1}$ exists,
we have
\begin{equation}
G_{E_{n}}=(H_{E_{n}}^{E_{n}})^{-1}=(I_{E_{n}}+K_{n})^{-1}L_{n}. \label{LnH}%
\end{equation}

A bound on the operator norm $\left\Vert L_{n}\right\Vert _{\sigma_{n}}$ uses
that $\left\Vert G_{A_{n}}\right\Vert _{\sigma_{n}}\leq c/d_{0}$ and
$\left\Vert P_{B_{n-1}}G_{C_{n}(B_{n-1})}\right\Vert _{\sigma_{n-1}}\leq
c_{\gamma_{n-1}}/d_{n-1}$. Then referring to Lemma \ref{Lem2}, we have
\[
\left\Vert \sum_{S_{j}\subset\mathcal{S}_{n}}P_{S_{j}}G_{C_{n}(S_{j}%
)}\right\Vert _{\sigma_{n-1}-\gamma_{n}}\leq c_{\gamma_{n}}\sup_{x\in
\mathcal{S}_{n}}\left\Vert G_{C_{n}(x)}\right\Vert _{\sigma_{n-1}}\lesssim
c_{\gamma_{n}}/d_{n}\text{.}%
\]
Since $L_{n}$ is the sum of the above three operators, we conclude that
\[
\left\Vert L_{n}\right\Vert _{\sigma_{n}}\leq c/d_{0}+c_{\gamma_{n-1}}%
/d_{n-1}+c_{\gamma_{n}}/d_{n}\lesssim c_{\gamma_{n}}/d_{n}.
\]

There is a recursive relation among the constants $\gamma_{n}$, $\ell_{n}$ and
$d_{n}$, for which we will show that $\left\Vert K_{n}\right\Vert _{\sigma
_{n}}\leq3/4$, hence $\left\Vert I+K_{n}\right\Vert _{\sigma_{n}}\leq4$.
Therefore, we obtain the result of the theorem%
\[
\left\Vert G_{E_{n}}\right\Vert _{\sigma_{n}}\leq4\left\Vert L_{n}%
^{-1}\right\Vert _{\sigma_{n}}\lesssim c_{\gamma_{n}}/d_{n}.
\]

To estimate $K_{n}$, we rewrite it as%
\[
K_{n}=G_{A_{n}}T_{A_{n}}^{E_{n}\backslash A_{n}}+P_{B_{n-1}}G_{C_{n}(B_{n-1}%
)}TP_{E_{n}\backslash C_{n}(B_{n-1})}+\sum_{S_{j}\subset\mathcal{S}_{n}%
}P_{_{S_{j}}}G_{C_{n}(S_{j})}TP_{E_{n}\backslash C_{n}(S_{j})}\text{.}%
\]
The proof exploits the fact that $\hbox{\rm dist}(B_{n-1},E_{n}\backslash
C_{n}(B_{n-1}))$ and $\hbox{\rm dist}\left(  S_{j},E_{n}\backslash C_{n}%
(S_{j})\right)  $ are greater than $\ell_{n}$, i.e. the long step $\ell_{n}$
gives the estimate in the exponential decay of the norm $\sigma_{n-1}%
-\gamma_{n}$ by $e^{-\gamma_{n}\ell_{n}}$. That is, by the first result of
Lemma \ref{Lem2} we have for $S_{j}$ that%
\[
\left\Vert P_{_{S_{j}}}G_{C_{n}(S_{j})}TP_{E_{n}\backslash C_{n}(S_{j}%
)}\right\Vert _{\sigma_{n-1}-\gamma_{n}}\leq\frac{c_{\gamma_{n}}}{d_{n}}%
r_{0}e^{-\gamma_{n}\ell_{n}}\text{,}%
\]
while for $B_{n-1}$ we have
\[
\left\Vert P_{B_{n-1}}G_{C_{n}(B_{n-1})}TP_{E_{n}\backslash C_{n}(B_{n-1}%
)}\right\Vert _{\sigma_{n-1}-\gamma_{n}}\leq\frac{c_{\gamma_{n-1}}}{d_{n-1}%
}r_{0}e^{-\gamma_{n}\ell_{n}}\text{.}%
\]
By the second result of Lemma \ref{Lem2}, we can estimate row by row the
matrix involving the singular sites $\mathcal{S}_{n}$ as
\[
\left\Vert \sum_{S_{j}\subset\mathcal{S}_{n}}P_{_{S_{j}}}G_{C_{n}(S_{j}%
)}TP_{E_{n}\backslash C_{n}(S_{j})}\right\Vert _{\sigma_{n-1}-2\gamma_{n}}%
\leq\frac{c_{\gamma_{n}}}{d_{n}}r_{0}e^{-\gamma_{n}\ell_{n}}\text{.}%
\]
Thus we conclude that
\[
\left\Vert K_{n}\right\Vert _{\sigma_{n-1}-2\gamma_{n}}\leq\frac{c}{d_{0}%
}r_{0}+\frac{c_{\gamma_{n-1}}}{d_{n-1}}r_{0}e^{-\gamma_{n}\ell_{n}}%
+\frac{c_{\gamma_{n}}}{d_{n}}r_{0}e^{-\gamma_{n}\ell_{n}}\leq\frac{3}{4}%
\]
since we have assumed that $cr_{0}/d_{0}<1/4$ and $c_{\gamma_{n}}%
r_{0}e^{-\gamma_{n}\ell_{n} }/d_{n}<1/4$.
\end{proof}

The inductive procedure of the estimate $G_{B_{n}}(w_{n-1})$ assumes the
analogous estimate for $G_{B_{n-1}}(w_{n-2})$ holds. But in the proof, the
estimate for $G_{B_{n-1}}(w_{n-1})$ is used instead. In the next lemma we
prove that this estimate is true also because $w_{n-1} - w_{n-2} = \delta
w_{n-1}$ is bounded for all $n$ in the Nash-Moser procedure.

\begin{lemma}
\label{Lem3}For any set $E_{n-1}=B_{n-1}\cup A_{n-1}$ with $A_{n-1}$ regular
we have that
\[
\left\Vert G_{E_{n-1}}(w_{n-1})\right\Vert _{\sigma_{n-1}}\lesssim
c_{\gamma_{n-1}}/d_{n-1}\text{.}%
\]

\end{lemma}

\begin{proof}
We assume from the previous step that
\[
\left\Vert G_{E_{n-1}}(w_{n-2})\right\Vert _{\sigma_{n-1}}\lesssim
c_{\gamma_{n-1}}/d_{n-1}\text{.}%
\]
The difference of the Hamiltonians is defined as%
\begin{align*}
R_{n-1} &  :=H_{E_{n-1}}(w_{n-1})-H_{E_{n-1}}(w_{n-2})\\
&  =T_{E_{n-1}}(w_{n-1})-T_{E_{n-1}}(w_{n-2}).
\end{align*}
From Taylor's theorem, since $\delta w_{n-1}=w_{n-1}-w_{n-2}$, we have%
\[
R_{n-1}=T_{E_{n-1}}^{\prime}(w_{n-2})[\theta(\delta w_{n-1})].
\]
Since $\left\Vert T^{\prime}(w_{n-2})\right\Vert _{\sigma_{n-1}}%
\lesssim\left\Vert w_{n-2}\right\Vert _{\sigma_{n-1}}\left\vert v\right\vert
^{2}$ with $w_{n-2}$ bounded for all $n$ and $\left\vert v\right\vert \leq
r_{0}$, then
\[
\left\Vert R_{n-1}\right\Vert _{\sigma_{n-1}}\lesssim\left\Vert \delta
w_{n-1}\right\Vert _{\sigma_{n-1}}\lesssim\varepsilon_{n-1}.
\]
By the inductive hypothesis, we have that
\[
\left\Vert G_{E_{n-1}}(w_{n-2})R_{n-1}\right\Vert _{\sigma_{n-1}}%
\lesssim\varepsilon_{n-1}/d_{n-1}\leq1/2\text{,}%
\]
the result follows from the fact that
\[
G_{E_{n-1}}(w_{n-1})=G_{E_{n-1}}(w_{n-2})(I_{E_{n-1}}+G_{E_{n-1}}%
(w_{n-2})R_{n-1})^{-1}\text{.}%
\]

\end{proof}

\section{Verification of hypotheses}

In this section we prove the exponential decay of the T\"oplitz matrix $T$,
and we discuss the separation property of the singular regions. Then, we prove
the estimate of the spectrum of the Hamiltonians in the singular regions for
good parameters $(r,\Omega)$ in a subset $\mathcal{N}_{n}$.

\subsection{(h1) Exponential decay}

\begin{lemma}
If $\left\Vert u\right\Vert _{\sigma}<r_{0}$, then $\left\Vert T(u)\right\Vert
_{\sigma}<Cr_{0}.$
\end{lemma}

\begin{proof}
Let $u$ be a function with $\left\Vert u\right\Vert _{\sigma}<r_{0}$, by the
algebra property of the norm, we have that the function
\[
h\left(  u\right)  =\omega\underset{n=2}{\overset{\infty}{%
{\displaystyle\sum}
}}n\left(  -1\right)  ^{n}\bar{u}^{n-1}%
\]
satisfy $\left\Vert h\right\Vert _{\sigma}<Cr_{0}$. Then by definition
\[
\left\vert \left\langle h,e_{j,k,l}\right\rangle \right\vert <Cr_{0}%
e^{-\sigma|(j,k)|}\left\langle j,k\right\rangle ^{-s}.
\]

Let $x_{n}=\left(  j_{n},k_{n},l_{n}\right)  \in\Lambda$, then%
\[
T\left(  x_{1},x_{2}\right)  =\left\langle he_{x_{2}},e_{x_{1}}\right\rangle =%
{\displaystyle\sum\limits_{x_{3}\in\Lambda}}
\left\langle h,e_{x_{3}}\right\rangle \left\langle e_{x_{2}}e_{x_{3}}%
,e_{x_{1}}\right\rangle
\]
Since $\left\langle e_{x_{2}}e_{x_{3}},e_{x_{1}}\right\rangle =0$ when
$j_{3}\notin\{\pm j_{1}\pm j_{2}\}$ or $k_{3}\notin\{\pm k_{1}\pm k_{2}\}$,
and since $\left\vert \left\langle e_{x_{2}}e_{x_{3}},e_{x_{1}}\right\rangle
\right\vert \leq1$, then%
\begin{align*}
\left\vert T\left(  x_{1},x_{2}\right)  \right\vert  &  \leq\sum_{l_{3}%
=\pm1,~j_{3}\in\{\pm j_{1}\pm j_{2}\},~k_{3}\in\{\pm k_{1}\pm k_{2}%
\}}\left\vert \left\langle h,e_{j_{3},k_{3},l_{3}}\right\rangle \right\vert \\
&  \leq2\sum_{j_{3}\in\{\pm j_{1}\pm j_{2}\},~k_{3}\in\{\pm k_{1}\pm k_{2}%
\}}Cr_{0}e^{-\sigma|(j_{3},k_{3})|}\left\langle j_{3},k_{3}\right\rangle
^{-s}\text{.}%
\end{align*}
Since there are four elements in the sum, and they satisfy $j_{3}%
\geq\left\vert j_{1}-j_{2}\right\vert $ and $k_{3}\geq\left\vert k_{1}%
-k_{2}\right\vert $, then
\[
\left\vert T\left(  x_{1},x_{2}\right)  \right\vert \leq8Cr_{0}e^{-\sigma
|(\left\vert j_{1}-j_{2}\right\vert ,\left\vert k_{1}-k_{2}\right\vert
)|}\left\langle \left\vert j_{1}-j_{2}\right\vert ,\left\vert k_{1}%
-k_{2}\right\vert \right\rangle ^{-s}\text{.}%
\]

\end{proof}

\subsection{(h2) Separation property}

We say that $(j,k,l)\in\Lambda$ is a singular site if $\left\vert
\lambda_{j,k,l}(\Omega)\right\vert \leq d_{0}$.

\begin{lemma}
Let $(j_{1},k_{1})$ and $(j_{2},k_{2})$ be two different singular sites, then
for a sufficiently small $d_{0}$, we have that
\[
\left\vert j_{1}-j_{2}\right\vert \geq C\left\vert k_{1}+k_{2}\right\vert
\text{,}%
\]
where $C$ is a constant that only depends on $\Omega$ and $\omega$.
Furthermore, the constant is uniform in $(j,k,l) \in\Lambda$ for
$(\Omega,\omega)$ in neighborhood of $(\Omega_{0},\omega_{0})$.
\end{lemma}

\begin{proof}
The sites of the form $(j,k,1)$ are never singular if $d_{0}<<1$. Given that
$\lambda_{j,k,-1}=k^{2}+\omega-\sqrt{(j\Omega)^{2}+\omega^{2}}$, then
\[
\left\vert k_{1}^{2}-k_{2}^{2}\right\vert -C\Omega\left\vert j_{1}%
-j_{2}\right\vert \leq\left\vert \lambda_{j_{1},k_{1},-1}-\lambda_{j_{2}%
,k_{2},-1}\right\vert \leq2d_{0}.
\]
If $k_{1}=k_{2}$, taking $d_{0}$ small enough such that $d_{0}\leq C\Omega/2$,
then $j_{1}=j_{2}$. Finally, if $k_{1}\neq k_{2}$, there exists a constant $c$
such that
\[
\left\vert j_{1}-j_{2}\right\vert \geq\frac{1}{C\Omega}\left(  \left\vert
k_{1}^{2}-k_{2}^{2}\right\vert -2d_{0}\right)  \geq c\left\vert k_{1}%
+k_{2}\right\vert .
\]

Now let $S=\{(j_{0},k_{0},-1)\}$ be a singular site in the annulus
$B_{n+1}/B_{n}$. By the previous inequality, the neighborhood
\[
C_{n}(S)=\{(j,k,l):\left\vert (j,k)-(j_{0},k_{0})\right\vert <\ell_{n}\}
\]
contains only one singular site for $\ell_{n}=CL_{n}^{1/2}$.
\end{proof}

\subsection{(h3) Good parameters}

We analyze the spectrum of the local Hamiltonians%
\[
H_{C(S)}(w_{n};r,\Omega)=P_{C(S)}(D(\Omega)+T(v(r)+w_{n}(r,\Omega
);\Omega))P_{C(S)}\text{.}%
\]

For $\left\vert r\right\vert \leq r_{0}\ $, there is only one eigenvalue of
$H_{C(S)}(w_{n},\Omega)$ with norm less than $d_{0}/2$ if $r_{0}<<d_{0}/2$.
Let $e(r,\Omega)$ be that eigenvalue of $H_{C(S)}$ with
\[
\left\vert e(r,\Omega)\right\vert <d_{n}<<d_{0}/2,
\]
then $e(r,\Omega)$ is isolated from other eigenvalues, and as such it is
analytic in the parameters.

\begin{lemma}
For $\left\vert r\right\vert <r_{0}$, there exists a constant $C>0$ such that
\[
\partial_{\Omega}e(r,\Omega)\leq-CL_{n}\text{.}%
\]

\end{lemma}

\begin{proof}
Since $e(r,\Omega)$ is analytic, then
\[
\partial_{\Omega}e(r,\Omega)=\partial_{\Omega}e(0,\Omega)+O(r)=\partial
_{\Omega}\lambda_{j,k,-1}+O(r)\text{.}%
\]
By an explicitly calculation
\[
\partial_{\Omega}\lambda_{j,k,-1}=\frac{-j^{2}\Omega}{\sqrt{j^{2}\Omega
^{2}+\omega^{2}}}\geq-C\Omega L_{n}\text{,}%
\]
because the site $(j,k)$ is singular with $\left\vert j\right\vert \geq
CL_{n}$.

From the above lemma, for a fix $r<r_{0}$, the eigenvalue $e(r,\Omega)$ is a
monotone decreasing function of $\Omega$. Since $e(r,\Omega)$ is analytic,
then there is an unique analytic function $\Omega_{z}(r)$ such that
$e(r,\Omega_{z}(r))=0$. Since $e(0,\Omega)=\lambda_{j,k,-1}(\Omega)$ for
$r=0$, then
\[
\Omega_{z}(0)=\Omega_{j,k}=\frac{1}{j}\sqrt{k^{4}+2k^{2}\omega}%
\]
and $(j,k,-1)\in S_{j}$.
\end{proof}

\begin{lemma}
We have that%
\[
\Omega_{z}(r)=\Omega_{j,k}+\frac{1}{L_{n}}O(r^{2})\text{.}%
\]

\end{lemma}

\begin{proof}
Since $e(r,\Omega)$ is analytic then%
\[
e(r,\Omega)=e(0,\Omega_{j,k})+\partial_{\Omega}e(0,\Omega_{j,k})(\Omega
-\Omega_{j,k})+\partial_{r}e(0,\Omega_{j,k})r+h.o.t
\]
By the Feynman-Hellman formula, we have that%
\[
\partial_{r}e(0,\Omega_{j,k})=\langle T^{\prime}(0)[\partial_{r}v,\psi
_{0}],\psi_{0}\rangle\text{.}%
\]
In the space $L_{sym}^{2}$, the functions are $\psi_{0}=e_{j,k,-1}$,
$\partial_{r}v=e_{1,1,-1}$, and
\[
d^{2}g(0)[w_{1},w_{2}]=\partial_{\bar{u}}^{2}g(0)\bar{w}_{1}\bar{w}_{2}%
=2\bar{w}_{1}\bar{w}_{2}.
\]
Thus, for any $j,k\neq1/2$, we have that
\[
\langle\psi_{0},T^{\prime}(0)[\partial_{r}v,\psi_{0}]\rangle=\langle
2\omega\bar{e}_{1,1,-1}\bar{e}_{j,k,-1},e_{j,k,-1}\rangle=0.
\]

Since $\partial_{\Omega}e(0,\Omega_{j,k})<-CL_{n}$, then
\[
\partial_{\Omega}e(0,\Omega_{j,k})(\Omega_{z}-\Omega_{j,k})+h.o.t.=0
\]
Using the implicit function theorem, we have that $\Omega_{z}-\Omega_{j,k}$ is
a function of $r$, and%
\[
\Omega_{z}-\Omega_{j,k}=\frac{1}{L_{n}}O(r^{2})\text{.}%
\]

\end{proof}

Let $N_{j,k}$ be the neighborhood of the curve $\Omega_{z}(r)$ given by
\begin{equation}
N_{j,k}=\{(r,\Omega):\left\vert \Omega_{z}(r)-\Omega\right\vert <C\frac{d_{n}%
}{L_{n}}\}\text{,}%
\end{equation}
by the previous lemma, and the mean value theorem, the eigenvalue satisfy
$\left\vert e(r,\Omega)\right\vert >d_{n}$ if $(r,\Omega)\notin N_{j,k}$ .
Thus, the hypothesis (h3) holds true in the complement of the set of
parameters
\[
\mathcal{\cup}_{(j,k)\in\mathcal{S}_{n}}N_{j,k}\text{,}%
\]
where the union is taken over all singular sites in the annulus $B_{n}%
\backslash B_{n-1}$.

\section{Intersection property}

In this section, we present the arguments for the Whitney regularity of
$w(r,\Omega)$ for $(r,\Omega)\in\mathcal{N=}\cap_{n=1}^{\infty}\mathcal{N}%
_{n}$. Then, we prove that the intersection of the curve
\[
\mathcal{C=}\{(r,\Omega(r)):\Omega(r)=\Omega_{0}+\Omega_{2}r^{2}%
+O(r^{3})\mathcal{\}}\text{,}%
\]
and the Cantor set $\mathcal{N}$ has positive measure in the case $\Omega
_{2}\neq0$. Finally, we prove that the curve $\mathcal{C}$ corresponding to
the bifurcation of standing waves has the non-degeneracy property $\Omega
_{2}\neq0$.

\subsection{Whitney regularity}

At the $n^{th}$ Nash-Moser step excisions are made in parameter space
$(r,\Omega)\in\mathcal{N}_{n-1}$ consisting of the union of neighborhoods
$N_{j,k}$, each of width $Cd_{n}/L_{n}$. On the parameter set $\mathcal{N}%
_{n-1}\backslash\cup_{j,k}N_{j,k}$, after the excision of the $N_{j,k}$, we
solve \eqref{Eqn:ApproxNewtonIteration} for the correction $\delta
w_{n}(r,\Omega)$. We may now provide a smooth interpolant for $\delta
w_{n}(r,\Omega)$ across the excisions, in the usual way. Construct a cutoff
function $\varphi_{n}(r,\Omega)\in C_{0}^{\infty}$, which is supported in
$\mathcal{N}_{n-1}\backslash\cup_{j,k}N_{j,k}$ and for which $\varphi
_{n}(r,\Omega)=1$ on the new parameter set $\mathcal{N}_{n}:=\mathcal{N}%
_{n-1}\backslash\cup_{j,k}2N_{j,k}$, where
\[
2N_{j,k}:=\{(r,\Omega):\left\vert \Omega_{z}(r)-\Omega\right\vert
<c\frac{2d_{n}}{L_{n}}\}\text{{}}%
\]
are excisions of just twice the width of the previous $N_{j,k}$. This can be
done so that the cutoff function $\varphi_{n}$ has derivatives bounded by
$|\partial_{\Omega}^{\alpha}\partial_{r}^{\beta}\varphi_{n}|\leq C(\frac
{L_{n}}{2d_{n}})^{\alpha+\beta}$. Then $\varphi_{n}\delta w_{n}(r,\Omega)\in
C^{\infty}(\mathcal{N}_{0})$ and $\varphi_{n}\delta w_{n}=\delta w_{n}$ on
$\mathcal{N}_{n}$.
Now $w_{n}=w_{n-1}+\varphi_{n}\delta w_{n}$ is $C^{\infty}$ in the set of
parameters $(r,\Omega)\in\mathcal{N}_{0}$, $w_{n}=w_{n-1}+\delta w_{n}$ on
$\mathcal{N}_{n}$, and moreover, for $(r,\Omega)\in\mathcal{N}_{0}$ the
sequence $w_{n}$ converges in $h_{\sigma_{0}/2}$ along with all of its
derivatives with respect to $(r,\Omega)$; and the following estimate holds%
\[
\left\Vert \partial_{\Omega}^{\alpha}w\right\Vert _{\sigma_{0}/2}\leq
Cr^{2},\qquad\left\Vert \partial_{r}\partial_{\Omega}^{\alpha}w\right\Vert
_{\sigma_{0}/2}\leq Cr\text{.}%
\]

\subsection{Measure of good parameters}

\begin{lemma}
Let $r_{-}$ and $r_{+}$ be the minimum and the maximum of $\{r:(r,\Omega
)\in\mathcal{C}\cap N_{j,k}\}$, we have that
\[
\left\vert r_{-}^{2}-r_{+}^{2}\right\vert \leq\frac{C}{\Omega_{2}}\frac{d_{n}%
}{L_{n}}\text{.}%
\]

\end{lemma}

\begin{proof}
Let $r_{0}$ be such that $\Omega_{z}(r_{0})=\Omega(r_{0})$, then the point
$(r_{0},\Omega(r_{0}))$ is the intersection of the curves $\Omega_{z}(r)$ and
$\Omega(r)$. Since $\Omega(r)=\Omega_{0}+\Omega_{2}r^{2}+O(r^{3})$, then%
\[
\left\vert \Omega(r_{-})-\Omega(r_{0})\right\vert \geqslant\frac{\Omega_{2}%
}{2}|r_{-}^{2}-r_{0}^{2}|\text{.}%
\]
By the previous lemma
\[
\left\vert \Omega_{z}(r_{0})-\Omega_{z}(r_{-})\right\vert \leq\frac{C}{L_{n}%
}\left\vert r_{-}^{2}-r_{0}^{2}\right\vert \text{.}%
\]

Since $\Omega(r_{-}),\Omega_{z}(r_{-})\in N_{j,k}$, then
\begin{align*}
\frac{Cd_{n}}{L_{n}}  &  \geq\left\vert \Omega(r_{-})-\Omega_{z}%
(r_{-})\right\vert \geq\left\vert \Omega(r_{-})-\Omega_{z}(r_{0})\right\vert
-\left\vert \Omega_{z}(r_{0})-\Omega_{z}(r_{-})\right\vert \\
&  \geq\frac{\Omega_{2}}{2}|r_{-}^{2}-r_{0}^{2}|-\frac{C}{L_{n}}\left\vert
r_{-}^{2}-r_{0}^{2}\right\vert \geq\frac{\Omega_{2}}{4}|r_{-}^{2}-r_{0}%
^{2}|\text{.}%
\end{align*}
Analogously, we have for the estimate of $r_{+}$ that $\left\vert r_{+}%
^{2}-r_{0}^{2}\right\vert \leq\frac{Cd_{n}}{\Omega_{2}L_{n}}$. The lemma
follows from the triangle inequality.
\end{proof}

\begin{lemma}
If $\Omega_{2}\neq0$, the measure of the set $\{r:(r,\Omega)\in\mathcal{C\cap
}N_{j,k}\}$ is bounded by
\[
\left\vert \{r:(r,\Omega)\in\mathcal{C\cap}N_{j,k}\}\right\vert <\frac
{C}{\sqrt{\Omega_{2}}}\frac{d_{n}}{\sqrt{L_{n}}}\text{.}%
\]

\end{lemma}

\begin{proof}
For a singular site $\Omega_{0}\left\vert j\right\vert \sim k^{2}$, then
\[
\Omega_{z}(0)=\Omega_{j,k}=\left\vert k/j\right\vert \sqrt{k^{2}+2\omega}%
\sim\Omega_{0}\sqrt{1+2\omega k^{-2}}.
\]
Thus,
\[
\left\vert \Omega_{0}-\Omega_{z}(0)\right\vert =\Omega_{0}\left\vert
\frac{-2\omega k^{-2}}{1+\sqrt{1+2\omega k^{-2}}}\right\vert >\Omega
_{0}\left\vert 2\omega k^{-2}\right\vert \gtrsim\frac{C}{L_{n}}.
\]

From the definition of $r_{-}$, we have $\left\vert \Omega_{z}(r_{-}%
)-\Omega(r_{-})\right\vert <Cd_{n}/L_{n}$. By the properties of $\Omega_{z}$
we have $\left\vert \Omega_{z}(0)-\Omega_{z}(r_{-})\right\vert <Cr_{-}%
^{2}/L_{n}$, then
\[
\frac{C}{L_{n}}r_{-}^{2}+\frac{Cd_{n}}{L_{n}}>\left\vert \Omega(r_{-}%
)-\Omega_{z}(0)\right\vert \text{.}%
\]
Since the curve $\mathcal{C}$ is of the form $\Omega(r)=\Omega_{0}+\Omega
_{2}r^{2}+O(r^{3})$, then
\[
\left\vert \Omega(r_{-})-\Omega_{z}(0)\right\vert >\left\vert \Omega
_{0}-\Omega_{z}(0)\right\vert -2\Omega_{2}r_{-}^{2}\text{.}%
\]
Thus,
\[
\left(  2\Omega_{2}+\frac{C}{L_{n}}\right)  r_{-}^{2}>\left\vert \Omega
_{0}-\Omega_{z}(0)\right\vert -\frac{Cd_{n}}{L_{n}}>\frac{C}{L_{n}}%
(1-Cd_{n})>\frac{C}{2L_{n}}\text{.}%
\]

For $\Omega_{2}\neq0$, we conclude that $r_{+}>r_{-}>C/\sqrt{\Omega_{2}L_{n}}
$. By the above lemma, we have%
\[
r_{+}-r_{-}\leq\frac{C}{\Omega_{2}}\frac{d_{n}}{L_{n}}(C\sqrt{\Omega_{2}L_{n}%
})\leq\frac{C}{\sqrt{\Omega_{2}}}\frac{d_{n}}{\sqrt{L_{n}}}.
\]

\end{proof}

\begin{proposition}
Let $d_{n}=L_{n}^{-\beta}$. If $\Omega_{2}\neq0$ and $\beta>3/2$, then the
measure of good parameters is positive. Moreover,
\[
\left\vert \{r\in\lbrack0,r_{0}):(r,\Omega)\in\mathcal{N}\cap\mathcal{C}%
\}\right\vert >r_{0}(1-r_{0}C_{\beta}),
\]
where
\[
C_{\beta}=\frac{C}{\sqrt{\Omega_{2}}}\sum_{n=1}^{\infty}L_{n}^{3/2-\beta
}<\infty\text{.}%
\]

\end{proposition}

\begin{proof}
There are at most $cL_{n}^{2}$ singular sites at the $n$ step, then the
previous lemma implies%
\[
\left\vert \{r\in\lbrack0,r_{0}):(r,\Omega(r))\notin\mathcal{N}_{n}%
\cap\mathcal{C}\}\right\vert \leq r_{0}L_{n}^{2}\frac{C}{\sqrt{\Omega_{2}}%
}\frac{d_{n}}{\sqrt{L_{n}}}\text{.}%
\]
Thus%
\[
\left\vert \{r\in\lbrack0,r_{0}):(r,\Omega)\in\mathcal{N}\cap\mathcal{C}%
\}\right\vert \geq r_{0}-r_{0}\frac{C}{\sqrt{\Omega_{2}}}\sum_{n=1}^{\infty
}L_{n}^{3/2-\beta}=r_{0}(1-r_{0}C_{\beta})\text{.}%
\]

\end{proof}

\subsection{Non-degeneracy of the bifurcation branch}

In this section we prove that for all parameters $\omega$ with only one
exceptional value $\omega= \omega_{0}$ we have $\Omega_{2}\neq0$; namely, the
bifurcation branch has nondegenerate curvature at the bifurcation point. For
the asymptotic expansion
\begin{align*}
u  &  =ru_{1}+r^{2}u_{2}+\mathcal{O}(r^{3}),\text{ }\\
\Omega(r)  &  =\Omega_{0}+r\Omega_{1}+r^{2}\Omega_{2}+\mathcal{O}%
(r^{3})\text{,}%
\end{align*}
we have that%
\[
f(w;\Omega)=Lu+i(\Omega(r)-\Omega_{0})u_{t}+\omega\bar{u}^{2}-\omega\bar
{u}^{3}+\mathcal{O}(\left\vert u\right\vert ^{4})=0\text{,}%
\]
where $L=L(\Omega_{0})$ is the linear map at $\Omega_{0}$.

At order $r$, we have $Lu_{1}=0$, and then $u_{1}=e_{1,1,-1}(\Omega_{0})$.
Thus, at order $r^{2}$ we have%
\[
Lu_{2}-i\Omega_{1}\partial_{t}u_{1}+\omega\overline{u}_{1}^{2}=0\text{.}%
\]
Multiplying by $u_{1}$, integrating by parts, and using that $L$ is
self-adjoint with $Lu_{1}=0$, we get that%
\[
\Omega_{1}\left\langle i\partial_{t}u_{1},u_{1}\right\rangle =\omega
\left\langle \bar{u}_{1}^{2},u_{1}\right\rangle .
\]

The basis $e_{j,k,l}$ at $\Omega=\Omega_{0}$ is given by $e_{0,0,1}=1$,%
\begin{align*}
e_{0,2,1}  &  =\sqrt{2}\cos2s\text{,}\\
e_{1,1,-1}  &  =2(a\cos t+ib\sin t)\cos s,\\
e_{2,0,\pm1}  &  =\sqrt{2}(a_{\pm}\cos2t+ib_{\pm}\sin2t)\text{,}\\
e_{2,2,\pm1}  &  =2(a_{\pm}\cos2t+ib_{\pm}\sin2t)\cos2s,
\end{align*}
where $(a,b)^{T}=v_{1,1,-1}$ and $(a_{\pm},b_{\pm})^{T}=v_{2,0,,\pm1}$. Since
\[
\left\langle i\partial_{t}u_{1},u_{1}\right\rangle =\frac{1}{4\pi^{2}}%
\int_{{\mathbb{T}}^{2}}i\partial_{t}u_{1}\bar{u}_{1}~dtds=-2ab\text{,}%
\]
then%
\begin{equation}
\left\langle \overline{u}_{1}^{2},u_{1}\right\rangle =\frac{1}{4\pi^{2}}%
\int_{{\mathbb{T}}^{2}}e_{1,1,-1}^{3}~dtds=0. \label{I}%
\end{equation}
Since $\Omega_{1}=0$ and $u_{2}=-\omega L^{-1}\overline{u_{1}}^{2}$, then we
conclude that%
\[
\bar{u}_{2}=-\omega L^{-1}u_{1}^{2}\text{.}%
\]

At order $r^{3}$, we obtain%
\[
Lu_{3}-i\Omega_{2}\partial_{t}u_{1}+2\omega\overline{u_{1}}\overline{u_{2}%
}-\omega\bar{u}_{1}^{3}=0\text{.}%
\]
Multiplying by $u_{1}$ and integrating by parts, then%
\begin{equation}
-\Omega_{2}\left\langle i\partial_{t}u,u_{1}\right\rangle =\omega\left\langle
2\omega\overline{u_{1}}L^{-1}(u_{1}^{2})+\bar{u}_{1}^{3},u_{1}\right\rangle
\text{.}\nonumber
\end{equation}
Thus, we have
\[
\Omega_{2}=\frac{\omega}{2ab}\left(  \left\langle \bar{u}_{1}^{3}%
,u_{1}\right\rangle +2\omega\left\langle L^{-1}u_{1}^{2},u_{1}^{2}%
\right\rangle \right)  \text{.}%
\]

To calculate the first product, we use that
\[
\int_{0}^{2\pi}\cos^{4}\theta~d\theta=3\pi/4,\qquad\int_{0}^{2\pi}\cos
^{2}\theta\sin^{2}\theta~d\theta=\pi/4,
\]
then%
\begin{align*}
\left\langle \bar{u}_{1}^{3},u_{1}\right\rangle  &  =\frac{4}{\pi^{2}}\left(
a^{4}\frac{3\pi}{4}-6a^{2}b^{2}\frac{\pi}{4}+b^{4}\frac{3\pi}{4}\right)
\frac{3\pi}{4}\\
&  =\frac{9}{4}\left(  a^{4}-2a^{2}b^{2}+b^{4}\right)  =\frac{9}{4}\left(
a^{2}-b^{2}\right)  ^{2}%
\end{align*}
Since%

\[
\binom{a}{b}=\frac{1}{\sqrt{2}\sqrt{1+\omega}}\binom{1}{\sqrt{1+2\omega}%
}\text{,}%
\]
then
\begin{equation}
\left(  a^{2}-b^{2}\right)  ^{2}=\frac{\omega^{2}}{(1+\omega)^{2}}\text{ and
}2ab=\frac{\sqrt{1+2\omega}}{1+\omega}\text{.}%
\end{equation}

We conclude from the next proposition that%

\[
\Omega_{2}=\frac{1}{6}\frac{\omega^{2}}{\left(  \omega+1\right)  \left(
\omega+2\right)  \sqrt{2\omega+1}}\left(  4\omega^{3}+29\omega^{2}%
+33\omega-6\right)  .
\]
Since $\Omega_{2}=0$ at only one point $\omega_{0}>0$, the curve has the
property for the intersection with the Cantor set $\mathcal{N}$ except for
$\omega_{0}$.

\begin{proposition}
We have that
\[
\left\langle L^{-1}u_{1}^{2},u_{1}^{2}\right\rangle =\frac{1}{24\left(
\omega+1\right)  ^{2}\left(  \omega+2\right)  }\left(  8\omega^{3}%
+31\omega^{2}+12\omega-12\right)  \allowbreak\text{.}%
\]

\end{proposition}

\begin{proof}
To calculate $\left\langle L^{-1}u_{1}^{2},u_{1}^{2}\right\rangle $, we use
the expression for $u_{1}^{2}$ given by
\[
u_{1}^{2}=e_{1,1,-1}^{2}=(a^{2}-b^{2}+\cos2t+i2ab\sin2t)(1+\cos2s)\text{.}%
\]
Projecting the vector $(1,2ab)$ in the orthonormal components $(a_{+},b_{+})$
and $(a_{-},b_{-})$, we have that%
\begin{align*}
u_{1}^{2}  &  =\left(  a^{2}-b^{2}\right)  \left(  e_{0,0,1}+\frac{1}{\sqrt
{2}}e_{0,2,1}\right) \\
&  +\left(  a_{+}+2abb_{+}\right)  \left(  \frac{1}{\sqrt{2}}e_{2,0,1}%
+\frac{1}{2}e_{2,2,1}\right) \\
&  +\left(  a_{-}+2abb_{-}\right)  \left(  \frac{1}{\sqrt{2}}e_{2,0,-1}%
+\frac{1}{2}e_{2,2,-1}\right)  .
\end{align*}

Thus, we have
\begin{equation}
\left\langle L^{-1}u_{1}^{2},u_{1}^{2}\right\rangle =\left(  a^{2}%
-b^{2}\right)  ^{2}\left(  \lambda_{0,0,1}^{-1}+\frac{1}{2}\lambda
_{0,2,1}^{-1}\right)  +P\text{,}%
\end{equation}
where $P$ is the polynomial%
\begin{align*}
P  &  =\left(  a_{+}+2abb_{+}\right)  ^{2}\left(  \frac{1}{2}\lambda
_{2,0,1}^{-1}+\frac{1}{4}\lambda_{2,2,1}^{-1}\right) \\
&  +\left(  a_{-}+2abb_{-}\right)  ^{2}\left(  \frac{1}{2}\lambda
_{2,0,-1}^{-1}+\frac{1}{4}\lambda_{2,2,-1}^{-1}\right)  \text{.}%
\end{align*}
For the first term we have
\[
\left(  \lambda_{0,0,1}^{-1}+\frac{1}{2}\lambda_{0,2,1}^{-1}\right)  \left(
a^{2}-b^{2}\right)  ^{2}=\frac{1}{4}\frac{\omega}{\left(  \omega+1\right)
^{2}\left(  \omega+2\right)  }\left(  3\omega+4\right)  \text{.}%
\]

To calculate $P$, we define the polynomial
\[
Q=\sqrt{\omega^{2}+8\omega+4}\text{,}%
\]
then $a_{\pm}$ and $b_{\pm}$ are given by%
\[
\binom{a_{\pm}}{b_{\pm}}=\frac{1}{\sqrt{2}\sqrt{Q^{2}\pm\omega Q}}%
\binom{\omega\pm Q}{-2\sqrt{1+2\omega}}\text{,}%
\]
and the eigenvalues are $\lambda_{0,0,1}=2\omega$, $\lambda_{0,2,1}%
=2(\omega+2)$, $\lambda_{2,0,\pm1}=\omega\pm Q$ and $\lambda_{2,2,\pm1}%
=\omega\pm Q+4$.

Thus, we have
\[
\frac{1}{2}\lambda_{2,0,\pm1}^{-1}+\frac{1}{4}\lambda_{2,2,\pm1}^{-1}=\frac
{1}{48}\frac{1}{2\omega+1}\left(  2\omega^{2}+3\omega+4\pm(5-2\omega)Q\right)
\text{,}%
\]
and%
\[
\left(  a_{\pm}+2abb_{\pm}\right)  ^{2}=\frac{1}{2(Q^{2}\pm\omega Q)}\left(
\omega\pm Q-2\frac{1+2\omega}{1+\omega}\right)  ^{2}\text{.}%
\]
After a computations with Maple, and alternatively by hand, we conclude that
\[
P=\frac{1}{24}\frac{1}{\left(  \omega+1\right)  ^{2}}\left(  8\omega
^{2}-3\omega-6\right)  \text{.}%
\]

\end{proof}

\section*{Appendix}

Periodic solutions of equation (\ref{pdev}) are zeros of the map
\[
f(v)=-i\Omega\partial_{t}v-\partial_{ss}v + \omega(1-\left\vert v\right\vert
^{-2})v\text{.}%
\]
The map $f$ is ${\mathbb{T}}^{3}$-equivariant with the action of
$(\theta,\varphi,\psi)\in\mathbb{T}^{3}$ given by
\[
\rho(\theta,\varphi,\psi)v=e^{i\theta}v(t+\varphi,s+\psi)\text{.}%
\]
In addition, the map is equivariant by the actions
\[
\rho(\kappa)v(t,s)=v(t,-s),\qquad\rho(\bar{\kappa})v(t,s)=\bar{v}(-t,s).
\]
Since
\[
\rho(\theta,\varphi,\psi)\rho(\kappa)=\rho(\kappa)\rho(\theta,\varphi
,-\psi),\qquad\rho(\theta,\varphi,\psi)\rho(\bar{\kappa})=\rho(\bar{\kappa
})\rho(-\theta,-\varphi,\psi)\text{,}%
\]
then the map $f$ is equivariant by the action of the non-abelian group
\[
\Gamma=O(2)\times(\mathbb{T}^{2}\cup\bar{\kappa}\mathbb{T}^{2}).
\]

The equilibrium $v_{0}$ is fixed by the actions of $\kappa$, $\bar{\kappa}$
and $(0,\varphi,\psi)\in\mathbb{T}^{3}$, then the isotropy group of $v_{0}$
is
\[
\Gamma_{v_{0}}=O(2)\times({\mathbb{S}}^{1}\cup\bar{\kappa}{\mathbb{S}}^{1}).
\]
The orbit of $v_{0}$ has dimension one, with
\[
\Gamma v_{0}=\{e^{i\theta}:\theta\in S^{1}\}.
\]

The linear map $f^{\prime}(v_{0})$ in coordinates $u=(v,\bar{v})\in
\mathbb{C}^{2}$ is given by
\[
L(\Omega)u=\left(
\begin{array}
[c]{cc}%
-i\Omega\partial_{t}-\partial_{ss}+\omega & \omega\\
\omega & -i\Omega\partial_{t}-\partial_{ss}+\omega
\end{array}
\right)  u\text{.}%
\]
The linear map $L\ $ in Fourier components is
\[
L(\Omega)u=\sum_{(j,k)\in\mathbb{Z}^{2}}M_{j,k}(\Omega)u_{j,k}e^{i(jt+ks)}%
\text{,}%
\]
where the matrix $\ $%
\[
M_{j,k}(\Omega)=\left(
\begin{array}
[c]{cc}%
\Omega j+k^{2}+\omega & \omega\\
\omega & -\Omega j+k^{2}+\omega
\end{array}
\right)
\]
has eigenvalues $\lambda_{j,k,\pm1}(\Omega)$ as before.

For $j=1$, the eigenvalue $\lambda_{1,k,-1}(\Omega)$ is zero at
\[
\Omega_{0}=\left\vert k\right\vert \sqrt{k^{2}+2\omega}.
\]
Let $(a,b)$ be the eigenvector of $M_{1,k}(\Omega_{0})$ corresponding to the
eigenvalue $\lambda_{1,k,-1}$. Thus, the periodic functions
\[
v(z_{1},z_{2})=z_{1}ae^{i(ks+t)}+\bar{z}_{1}be^{-i(ks+t)}+z_{2}ae^{i(-ks+t)}%
+\bar{z}_{2}be^{-i(-ks+t)}%
\]
are in the kernel of $f^{\prime}(v_{0};\Omega_{k})$ for $(z_{1},z_{2}%
)\in\mathbb{C}^{2}$.

The action of $\Gamma_{v_{0}}$ in the parametrization of the kernel
$(z_{1},z_{2})$ is given by
\begin{align*}
\rho(\varphi,\psi)(z_{1},z_{2})  &  =e^{i\varphi}(e^{ik\psi}z_{1},e^{-ik\psi
}z_{2}),\\
\rho(\kappa)(z_{1},z_{2})  &  =(z_{2},z_{1}),\\
\rho(\bar{\kappa})(z_{1},z_{2})  &  =(\bar{z}_{2},\bar{z}_{1}).
\end{align*}

The isotropy groups of the action are inherited by the bifurcating solutions.
Using the elements $\kappa$ and $\varphi\in{\mathbb{S}}^{1}$, we may assume
without loss of generalization that the first coordinate is real and positive,
$z_{1}=a\in\mathbb{R}$, unless both coordinates are zero.

If $z_{2}=0$, then the isotropy group is generated by
\[
(0,\varphi,-\varphi/k)\in\mathbb{T}^{3}\ \text{and }\kappa\bar{\kappa}\text{.}%
\]
This isotropy group $\Gamma_{(a,0)}$ is isomorphic to $O(2)$, and the orbit of
$(a,0)$ contains a $2$-torus.

If $z_{2}\neq0$, using the action of $(0,\varphi,-\varphi/k)\in{\mathbb{T}%
}^{3}$, which fixes the first coordinate and acts by multiplying the second
one by $e^{2\varphi}$, we may assume that $z_{2}$ is real.

If $z_{2}=z_{1}$, then the isotropy group of $(a,a)$ is generated by
\[
(0,\pi,-\pi/k)\text{, }\kappa\text{ and }\bar{\kappa}\text{.}%
\]
This isotropy group $\Gamma_{(a,a)}$ is finite, and its orbit contains a
$3$-torus. In other cases the isotropy group is generated by $(0,\pi,-\pi/k)$.

\subsection*{Traveling waves}

Functions in the fixed point space of $\Gamma_{(a,0)}$ satisfy
\[
v(t,s)=\rho(0,\varphi,-\varphi/k)v(t,s)=v(t+\varphi,s-\varphi/k),
\]
thus they are of the form
\[
v(t,s)=\sum_{l\in\mathbb{Z}}v_{l,lk}e^{i(lt+lks)}\text{.}%
\]
In this case the non-zero Fourier components are in a line on the lattice.
Moreover, from $v(t,s)=\bar{v}(-t,-s)$, we have that $v_{l,lk}\in\mathbb{R}$.

Since $f$ is $\Gamma\times{\mathbb{S}}^{1}$-equivariant, the map $f$ sends the
fixed point space of $\Gamma_{(a,0)}$ into itself. The restriction of $L$ to
this subspace is given by
\[
L\left(  u\right)  =\sum_{l\in\mathbb{Z}}M_{l,kl}u_{l,lk}e^{i(lt+slk)},
\]
where $u_{j,k}=(v_{j,k},\bar{v}_{j,k})$. Moreover, for $l=1$, if $\omega$ is
irrational, the kernel at
\[
\Omega=\Omega_{0}:=\left\vert k\right\vert \sqrt{k^{2}+2\omega}%
\]
is one dimensional. Since $\lambda_{1,k,-1}(\Omega)$ changes sign at
$\Omega_{0}$, by arguments of topological bifurcation \cite{IzVi03}, there is
a global bifurcation of periodic solutions near $(v_{0},\Omega_{0})$.

\begin{theorem}
The map $f(v;\Omega)$ has a global bifurcation of periodic traveling wave
solutions from $(1,\Omega_{0})$. These are solutions for the filament problem
of the form
\begin{equation}
u_{j}(t,s)=a_{j}e^{i\omega t}v(\Omega t+ks), \label{solls}%
\end{equation}
where $v$ is a $2\pi$-periodic solution.
\end{theorem}

In \cite{BaMi12} there is a proof of existence of traveling waves which are
asymptotic to parallel filaments at infinity. Also, in \cite{GaIz12} there is
a proof of global bifurcations of periodic traveling waves for a polygonal
relative equilibrium. The proof in this reference uses equivariant degree
theory to deal with the symmetries.

\subsection*{Standing waves}

The group $\Gamma_{(a,a)}$ is generated by the elements $(0,\pi,-\pi/k)$,
$\kappa$ and $\bar{\kappa}$. Thus, functions with the isotropy group
$\Gamma_{(a,a)}$ satisfy%
\[
v(t,s)=v(t+\pi,s-\pi/k)=v(t,-s)=\bar{v}(-t,s) ~.
\]
These solutions are the standing waves constructed in this article.

Since $v_{0}=1$ is the only solution of the orbit $e^{i\theta}v_{0}$ that
satisfy the symmetries (\ref{sym}), then $f^{\prime}(v_{0})$ is not degenerate
when restricted to the subspace of functions (\ref{sym}). We have taken
advantage of this fact in our proof of the existence of standing waves.

In order to solve the small divisor problem, the frequency $\omega$ has been
chosen to be diophantine. In this case there are no resonant Fourier
components, and all branches from different $k$'s can be derived from the same
branch $k=1$ and the transformation $\tau^{-1}v(\tau^{2}t,\tau s)$ for some
$\tau$. Thus, the case $k=1$ that we have analyzed in the paper is the general one.

\subsection*{Relative equilibria}

To find other relative equilibria we may solve equation (\ref{ode}) by the
method of quadratures, as in the Kepler problem. In polar coordinates,
$v=re^{i\theta}$, the equation becomes $\partial_{s}\theta=c/r^{2}$ where $c$
is the angular momentum, and
\begin{equation}
\partial_{s}^{2}r-c^{2}r^{-3}+r^{-1}-\omega r=0\text{.} \label{er}%
\end{equation}
If $c=0$, then $\theta=\theta_{0}$ is constant. In this case, each filament
lies in a plane that contain the central axis. Since $\omega$ is positive,
there is a unique equilibrium at $r=\omega^{-1/2}$, this corresponds to the
$2\pi$-periodic solution (\ref{sol}) with $\sigma=0$. There are more bounded
solutions with $r(s)\rightarrow0$ when $s\rightarrow\pm\infty$, but these
solutions are less relevant for modeling because the filaments approach each
other and therefore exit from the regime of validity of the approximation.
There are other relevant solutions that are unbounded.

If $c\neq0$, the dependency of equations on $s$ can be eliminated as in the
Kepler problem. Changing variables as $\rho=r^{-1}$ (see \cite{MeHa91}), then
$\partial_{s}r=-c\partial_{\theta}\rho$ and $\partial_{ss}r=-c^{2}\rho
^{2}\partial_{\theta\theta}\rho$. Thus, the equation becomes%
\begin{equation}
c^{2}\partial_{\theta\theta}\rho+c^{2}\rho-\rho^{-1}+\omega\rho^{-3}=0 ~.
\label{r}%
\end{equation}
These equation may be integrated from $c^{2}(\partial_{\theta}\rho)^{2}%
+V(\rho)=E$ with the potential
\[
V(\rho)=c^{2}\rho^{2}-\ln\rho^{2}-\omega\rho^{-2}\text{.}%
\]
Equation (\ref{r}) have two equilibria for $\omega\in(0,1/4c^{2})$ given by%
\[
\rho_{\pm}^{2}=\frac{1}{2c^{2}}\left(  1\pm\sqrt{1-4c^{2}\omega}\right)
\text{.}%
\]
Since $\omega=-c^{2}\rho_{\pm}^{4}+\rho_{\pm}^{2}$ and $\theta(s)=c\rho_{\pm
}^{2}s$, then the equilibria $\rho_{\pm}$ correspond to solutions $u=\rho
_{\pm}^{-1}e^{i(\omega t+c\rho_{\pm}^{2}s)}$, which are the helix solutions
(\ref{sol}).

Moreover, the equilibrium $\rho_{+}$ is a minimum of $V$, and there are
periodic solutions close to $\rho_{+}$. The continuum of periodic solutions
near $\rho_{+}$ consists of helices that are in a bounded annulus. Actually,
the projection in$\ $the plane of these periodic solutions are asymptotically
ellipses. More complex solutions may be obtained using the Galilean
transformation applied to the helix-like solutions.

\medskip\textbf{Acknowledgement.} The authors are partially supported by the
Canada Research Chairs Program, the Fields Institute and McMaster University.
Also, C.~Garc\'{\i}a-Azpeitia is supported by CONACyT through grant No.
203926. W.~Craig is supported by NSERC through grant No. 238452--11. C.R.~Yang
is supported by N.C.T.S. and N.S.C. through grant No. NSC 102-2811-M-007-075.

\end{document}